\documentclass{amsart}

\usepackage{comment, hyperref}

\begin{document}

\newtheorem{theorem}{Theorem}[section]
\newtheorem{prop}[theorem]{Proposition}
\newtheorem{lemma}[theorem]{Lemma}
\newtheorem{cor}[theorem]{Corollary}
\newtheorem{conj}[theorem]{Conjecture}

\theoremstyle{definition}
\newtheorem{definition}[theorem]{Definition}
\newtheorem{rmk}[theorem]{Remark}
\newtheorem{eg}[theorem]{Example}
\newtheorem{qn}[theorem]{Question}
\newtheorem{defn}[theorem]{Definition}

\numberwithin{equation}{section}

\newcommand{\boundary}{\partial}
\newcommand{\cc}{{\mathbb C}}
\newcommand{\integers}{{\mathbb Z}}
\newcommand{\natls}{{\mathbb N}}
\newcommand{\ratls}{{\mathbb Q}}
\newcommand{\reals}{{\mathbb R}}
\newcommand{\R}{{\mathbb R}}
\newcommand{\q}{{\mathbb Q}}
\newcommand{\N}{{\mathbb N}}
\newcommand{\D}{{\mathbb D}}
\newcommand\AAA{{\mathcal A}}
\newcommand\BB{{\mathcal B}}
\newcommand\CC{{\mathcal C}}
\newcommand\DD{{\mathcal D}}
\newcommand\EE{{\mathcal E}}
\newcommand\FF{{\mathcal F}}
\newcommand\GG{{\mathcal G}}
\newcommand\HH{{\mathcal H}}
\newcommand\II{{\mathcal I}}
\newcommand\JJ{{\mathcal J}}
\newcommand\KK{{\mathcal K}}
\newcommand\LL{{\mathcal L}}
\newcommand\MM{{\mathcal M}}
\newcommand\NN{{\mathcal N}}
\newcommand\OO{{\mathcal O}}
\newcommand\PP{{\mathcal P}}
\newcommand\QQ{{\mathcal Q}}
\newcommand\RR{{\mathcal R}}
\newcommand\SSS{{\mathcal S}}
\newcommand\SP{{{\mathcal S}{\mathcal P}}}
\newcommand\SM{{{\mathcal S}{\mathcal M}}}
\newcommand\TT{{\mathcal T}}
\newcommand\UU{{\mathcal U}}
\newcommand\VV{{\mathcal V}}
\newcommand\WW{{\mathcal W}}
\newcommand\XX{{\mathcal X}}
\newcommand\YY{{\mathcal Y}}
\newcommand\HC{{{\mathcal H}{\mathcal C}}}
\newcommand\CT{{{\mathcal C}{\mathcal T}}}
\newcommand\QP{{{\mathcal Q}{\mathcal P}}}
\newcommand\aff{{{\mathcal A}{\mathcal F}}}
\newcommand\QPK{{{\mathcal Q}{\mathcal P}{\mathcal K}}}
\newcommand\CR{{{\mathcal C}{\mathcal R}}}
\newcommand\CA{{{\mathcal C}{\mathcal A}}}
\newcommand\AC{{{\mathcal A}{\mathcal C}}}
\newcommand\ACT{{{\mathcal A}{\mathcal C}{\mathcal T}}}
\newcommand\ZZ{{\mathcal Z}}
\newcommand{\Q}{{\mathbb Q}}
\newcommand\F{{\mathbb F}}
\newcommand{\Ga}{\Gamma}
\newcommand\Z{{\mathbb{Z}}}
\newcommand\C{{\mathbb{C}}}
\newcommand\St{{{\mathcal{S}}t}}
\newcommand\ZG{{\mathbb{Z}}G}
\newcommand\til{\widetilde}

\title[Homotopical Height]{Homotopical Height}

\author[I. Biswas]{Indranil Biswas}
\address{School of Mathematics, Tata Institute of Fundamental
Research, Homi Bhabha Road, Bombay 400005, India}
\email{indranil@math.tifr.res.in}

\author[M. Mj]{Mahan Mj}
\address{RKM Vivekananda University, Belur Math, WB 711202, 
India}
\email{mahan.mj@gmail.com; mahan@rkmvu.ac.in}

\author[D. Pancholi]{Dishant Pancholi}
\address{Chennai Mathematical Institute, H1, SIPCOT IT Park, Siruseri
Kelambakkam 603103,
India}
\email{dishant@cmi.ac.in}



\date{\today}

\thanks{The first author is supported by J. C. Bose Fellowship. Research of the
second author is partly supported by CEFIPRA Indo-French Research Grant 4301-1. The
third author would like to thank ICTP for hosting him during January 2013 where
a part of the work for this article was done.}

\begin{abstract} Given a group $G$ and a class of manifolds $\CC$ (e.g. symplectic, contact, K\"ahler etc), it is an old
problem to find a manifold $M_G \in \CC$ whose fundamental group is $G$.
This article refines it: for a group $G$ and a positive integer $r$ find $M_G \in \CC$ such that
$\pi_1(M_G)=G$  and $\pi_i(M_G)=0$ for $1<i<r$.    We thus
provide a unified point of view systematizing  known  and  new results in this direction for various different classes of manifolds.
The largest $r$ for which such an $M_G \in \CC$ can be found is called the homotopical height $ht_\CC(G)$. 
Homotopical height provides a dimensional obstruction to finding
a $K(G,1)$ space within the given class $\CC$, leading
to a hierarchy of these classes in terms of
``softness'' or ``hardness'' \`a la Gromov. We show that the classes of closed
contact, CR, and almost complex
manifolds as well as the class of (open) Stein manifolds are soft. 

The classes $\SP$ and $\CA$ of closed
symplectic and complex manifolds exhibit intermediate ``softness'' in the sense
that every finitely presented group $G$ can be realized as
the fundamental group of a manifold in $\SP$ and a manifold in $\CA$. For these classes,
$ht_\CC(G)$ provides a numerical invariant for   finitely presented groups. We give explicit computations of these invariants
for some standard finitely presented groups.

We  use the notion of homotopical height within the ``hard'' category of 
K\"ahler groups to obtain partial answers to questions of Toledo regarding 
second cohomology and second group cohomology of  K\"ahler groups. We also modify and generalize a construction due to Dimca, Papadima and Suciu
to give a potentially large class of projective groups
violating property FP. 
\end{abstract}

\maketitle

\tableofcontents

\section{Introduction} 
 Given a group $G$ and one's favorite class of manifolds $\CC$ (e.g. symplectic, contact, K\"ahler etc), it is an old
problem to find a manifold $M_G \in \CC$ whose fundamental group is $G$.
This article refines it: for a group $G$ and a positive integer $r$ find $M_G \in \CC$ such that
$\pi_1(M_G)=G$  and $\pi_i(M_G)=0$ for $1<i<r$.    This paper 
provides a unified point of view systematizing  known  and  new results in this direction for various different classes of manifolds.
The largest $r$ for which such an $M_G \in \CC$ can be found shall be  called the {\it homotopical height} $ht_\CC(G)$
which provides a numerical/dimensional answer to the following basic question:

\begin{qn} Given a finitely presented group $G$ and a class $\CC$  of smooth manifolds (e.g. symplectic, contact, K\"ahler etc), what is the obstruction to
constructing a 
$K(G,1)$ manifold within the class $\CC$? \label{q} \end{qn}

We now formally introduce the  notion of {\it homotopical height}.

\begin{definition} \label{htd} Let $G$ be a finitely presented group, and
let $\CC$ be a given class of smooth manifolds of dimension greater than zero.
We shall say that the $\CC$--homotopical height of $G$ is $-\infty$
if $G$ is not the fundamental group of some manifold in $\CC$. If $\pi_2(M)\,\not=
\,0$ for all $M \,\in\, \CC$ with $\pi_1(M)\,=\, G$, then we
shall say that the $\CC$--homotopical height of $G$ is $2$. We shall say that
the $\CC$--homotopical height of $G$ is greater than 
or equal to $n$ if there exists a manifold $M \,\in\, \CC$ such that
$\pi_1(M)\,=\, G$ and $\pi_i (M) \,=\, 0$ for every $1\,<\,i \,<\, n$. The
$\CC$--homotopical  height of $G$ is the maximum value of $n$ such that
the $\CC$--homotopical  height of $G$ is greater than or equal to $n$. 

We shall denote the $\CC$--homotopical  height  of $G$ as $ht_\CC (G)$.
\end{definition}

Note that if there exists a $K(G,1)$ space in $\CC$, then the $\CC$--homotopical
height  of $G$ is infinite. Thus the 
$\CC$--homotopical  height  provides a dimensional obstruction to the
construction of a $K(G,1)$ manifold within the class $\CC$.

In Definition \ref{htd}, the dimension restriction ($> 0$) on elements of $\CC$
is for technical reasons to do with the exceptional
nature of the trivial group; this is explained in Section \ref{sec1.6}.

\begin{definition} \label{htd2} A manifold $M\,\in\,\CC$ is said to {\bf realize}
the  $\CC$-homotopical height of $G$ if $\pi_1(M)\,=\,G$ and $\pi_i(M)\,=\,
0$ for all $1\,<\,i\,<\,ht_\CC (G)$. If a manifold $M\,\in\,\CC$ realizing
$ht_\CC (G)$ has the property that there is no manifold in $\CC$ of strictly 
smaller dimension realizing $ht_\CC (G)$, then the dimension of $M$ is called
the $\CC$--\textit{homotopical dimension}. The $\CC$--homotopical dimension
of $G$ will be denoted by $hdim_\CC (G)$.
\end{definition}

Any $K(G,1)$ manifold in $\CC$
realizes the $\CC$--homotopical  height (which is $\infty$). 
The notions of homotopical height and homotopical dimension provide, as we hope to show in this paper, an organizing principle and language to address
a number of existing problems.

Let $G$ be a finitely presented group.
\begin{enumerate}
\item $G$ is of type $FP_\infty$ if it admits a $K(G,1)$ space (in the category
of CW complexes) with finitely many cells in each dimension.

\item $G$ is of type $FP_k$ if it admits a $K(G,1)$ space (in the category of CW
complexes) with finitely many cells in each dimension less than $(k+1)$.

\item $G$ is of type $FP$ if it admits a finite $K(G,1)$ space, i.e., a $K(G,1)$ 
space with finitely many cells.
\end{enumerate}
(See \cite{brown}.)

 Our investigations lead us to the following gradation of categories $\CC$ in degrees
of hardness or softness (\`a la Gromov \cite{gro-pdr,gromov-sq, gro-sh}) viewed through the lens of homotopical height:
We shall say that a class $\CC$ is of
\begin{enumerate}
\item[Type 1:] if $ht_\CC(G) \geq k$ for all finitely presented groups of type
$FP_k$, in particular, $ht_\CC(G) \,=\, \infty$ for all 
finitely presented groups of type $FP_\infty$; 
\item[Type 2:] if $ht_\CC(\{1\}) \,=\, \infty$ for the trivial group;
\item[Type 3:] if $ht_\CC(G) \geq 0$ for all finitely presented groups; and 
\item[Type 4:] if $ht_\CC (G) \,=\, -\infty$ for some finitely presented group of type $FP_\infty$.
\end{enumerate}

Clearly classes $\CC$ of Type 1 are the softest and those of Type 4 are hardest
in the above list. However, as we shall see below, Type 2 and Type 3
are classes that exhibit intermediate behavior of different (and not quite
comparable) kinds. Note that a class $\CC$ is of Type 3 if and only if every
finitely presented group $G$ appears as the fundamental group of some manifold
in $\CC$.

\subsection{Motivational Problems} There were a number of problems and theorems
that led us to the formulation of Definition \ref{htd} and the consequent
``hard/soft'' framework. In the next subsection, we shall summarize the partial
answers/approaches we obtain in this paper to some of the questions below.

\begin{theorem}[\cite{gompf}]
Every finitely presented group is the fundamental group of a closed symplectic manifold. \label{gompf-pi1} \end{theorem}

In the language of this paper, Theorem \ref{gompf-pi1} says that the class 
$\SP$ of closed symplectic manifolds is of Type 3 (we return to this later).

\begin{theorem}[\cite{taubes}] Every finitely presented group is the fundamental
group of a closed complex manifold. \label{taubes-pi1}
\end{theorem}

Theorem \ref{taubes-pi1} says that the class $\CA$ of closed complex manifolds
is of Type 3. The following question asks if the category $\CA$ of complex
manifolds is ``soft'':

\begin{qn}[{\cite[p. 238]{chern-bk}}, \cite{gromov-sq, yau-qn}]\label{yau-qn}
Does every almost complex closed manifold of real dimension
greater than four admit a complex structure? \end{qn}

We adapt Question \ref{yau-qn} to the context of this paper and ask:
\begin{qn} \label{yau-qn-ht} For a finitely presented group $G$ of type $FP$, is $ht_\CA (G) = \infty$?
\end{qn}

We note in passing that a group that provides a negative answer to Question \ref{yau-qn-ht} will also 
furnish a counterexample  to Question \ref{yau-qn}.

The fundamental groups of compact K\"ahler manifolds (respectively,
complex projective manifolds) will be called K\"ahler groups
(respectively, projective groups).

The following problem, due to Koll\'ar, asks for a kind of softness for
projective groups. The class $\PP$ of complex projective manifolds
and the class $\KK$ of closed K\"ahler manifolds belong to Type 4 in the above classification. We think of
these classes as ``hard''. 

\begin{qn}[{\cite[Section 0.3.1]{kollar-shaf}}]\label{kollar-qn}
Is a projective group 
$G_1$ commensurable to a group $G$, 
admitting a $K(G, 1)$ space which is a smooth quasi-projective variety? 
\end{qn}

Dimca, Papadima and Suciu \cite{dps-bb} have furnished examples of 
finitely presented groups giving a negative answer to Question 
\ref{kollar-qn}. However, these examples are not of type $FP_k$ for some 
$k$ and hence do not satisfy $ht_{\SM}(G) \,=\, \infty$ even for the class $\SM$ 
of closed smooth manifolds. We therefore modify Question \ref{kollar-qn} as 
follows:

\begin{qn}\label{kollar-qn1}
\mbox{}
\begin{enumerate}
\item Let $G_1$ be a projective group of type $FP_\infty$. 
Is $G_1$ commensurable to a group $G$, 
admitting a $K(G, 1)$ space which is a smooth quasi-projective variety? 

\item In particular, if $G$ is a torsion-free  projective group
of type $FP$, does it admit a $K(G, 1)$ space which is a quasi-projective variety?

\item If $G_1$ is a projective group of type $FP_k$, is $G_1$ commensurable
to a group $G$ with $ht_\QP(G)\,\geq\, k$, where $\QP$ denotes the category
of smooth quasiprojective varieties.
\end{enumerate}
\end{qn}

The examples in \cite{dps-bb}  do not provide counterexamples to Question
\ref{kollar-qn1}(3). In the form Question\ref{kollar-qn1}(3), Koll\'ar's
question applies to every finitely presented projective group (take $k\,=\,2$).

The next set of questions  specifically pertain to the hard (i.e., Type 4)
class $\KK$ of  K\"ahler manifolds. Our framework of ``homotopical height''
gives an approach and partial positive answers to these. An outstanding
conjecture on K\"ahler groups, due to Carlson and Toledo, states the following:

\begin{conj}[\cite{kollar-shaf}]\label{tolcon}
Let $G$ be an infinite  K\"ahler group. Then there exists a finite index
subgroup $G_1$ of $G$ such that the second Betti
number $ b_2 (G_1)$ is positive. \end{conj}

We were led to the following question thanks to Domingo Toledo.

\begin{qn} \label{tol-h2} Let $G$ be a K\"ahler group such that
$H^2(G,\, \ZG) \,\neq\, 0$. Is $G$ the fundamental group of a compact Riemann surface?
\end{qn}

\subsection{Outline of the Paper and Results} Let $\CC$ be a subclass of the class of  (not necessarily closed)
smooth manifolds of {\em positive dimension}. The classes we shall be most
interested here are as follows, arranged roughly in decreasing order of
``softness''.

\begin{enumerate}
\item The class $\SM$ of closed smooth manifolds.
\item The class $ACR$ of closed almost CR manifolds.
\item The class $\AC$ of closed almost complex manifolds.
\item The class $\CT$ of closed contact manifolds.
\item The class $\CR$ of closed CR manifolds.
\item The class $\St$ of (open) Stein manifolds.
\item The class $\SP$ of closed symplectic manifolds.
\item The class $\CA$ of closed complex manifolds.
\item The class $\aff$ of smooth affine varieties.
\item The class $\QP$ of smooth quasi-projective varieties.
\item The class $\KK$ of compact K\"ahler manifolds.
\item The class $\PP$ of smooth complex projective manifolds.
\item The subclass $\HC$ of $\PP$ consisting of smooth complex projective manifolds with holomorphically convex universal cover.
\item The subclass $\SSS$ of $\PP$ consisting of smooth complex projective manifolds with Stein universal cover.
\end{enumerate}

Other natural classes to which Definition \ref{htd} can and do apply will be defined as we go along.
This paper has roughly three thematic parts: \\
1) Section \ref{soft} deals with the ``soft'' categories: $\CT\, , AC\, R,  \AC
\, , \CR\, , \St$. \\
2) Section \ref{casp} deals primarily with the ``intermediate'' categories: $\SP\, , \CA$.\\
3) The later sections deal primarily with the ``hard'' categories $\KK$ and $\PP$. \\

The {\it techniques} used in these three types of categories are  different but not unrelated. We give below a representative selection of results
from each of these three  categories.

\medskip

\noindent {\bf Soft Categories:}\\
The main theorem here on the ``soft category'' proves softness (Type 1) for the classes $\CT , ACR,  \AC, \CR, \St$
(see Theorems \ref{sm-soft}, \ref{ac-soft},  \ref{stein-soft}, \ref{ct-soft}).

\begin{theorem} The classes $\SM, \CT ,   \AC, \CR, \St$ are each of Type 1.
In particular, $ht_\CC(G) \,=\, \infty$ for all groups $G$ of type $FP_\infty$
and $\CC = \SM, \CT,   \AC, \CR$ or $ \St$. \label{omni-soft} \end{theorem}

Softness of the class of closed almost complex manifolds $\AC$ provides a possible approach towards a {\it negative} answer to Question
\ref{yau-qn}. A group providing a negative answer to Question
\ref{yau-qn-ht} or to the more general Question \ref{yau-qn-ht2} below will immediately furnish an almost complex manifold which does not admit
a complex structure.

\medskip

\noindent {\bf Intermediate Categories:}\\
The classes $\SP$ and $\CA$ exhibit intermediate ``softness'' in the sense that every finitely presented group $G$
can be realized as the fundamental group of a manifold in $\SP$ and also
a manifold in $\CA$ (by Theorem \ref{gompf-pi1} and Theorem \ref{taubes-pi1} respectively).
For these classes,
$ht_\CC(G)$ provides a numerical invariant for finitely presented groups $G$. In Section \ref{casp},
we give explicit computations of these invariants
for some
standard finitely presented groups. 

The proof of the following theorem is in
Propositions \ref{ht-ca-eg}, \ref{ht-ca-finite}, \ref{ht-ca-nacs},
\ref{b20-sp}.

\begin{theorem}
We have $ht_\CA (G) = \infty$ for
\begin{enumerate}
\item[(a)] any finitely generated abelian group $G$,
\item[(b)] for any  finite group $G$, and
\item[(c)] the (integral) Heisenberg group $H^{2m+1}$.
\end{enumerate}

Let $G$ be a finitely presented group such that $b_2(G) = 0$. Then $ht_\SP(G) = 2$.
In particular, all finite groups and finitely generated free groups have $ht_\SP(G) = 2$.
\end{theorem}

We also provide some weak positive evidence towards Conjecture 
\ref{tolcon} (see Proposition \ref{toledob2}):

\begin{prop}  Let $G$ be 
 an infinite  finitely presented group (e.g. an infinite K\"ahler group) with $ht_\SP (G) > 2$. Then $ b_2 (G)$ is positive.\end{prop}

It follows that K\"ahler symplectically aspherical groups of class $\mathcal 
A$ in the terminology of \cite{ikrt} verify Conjecture 
\ref{tolcon}.

\medskip

\noindent {\bf Hard Categories:}\\ We turn to the harder categories $\KK, \PP, \HC, \SSS$ in Sections \ref{kk1}, \ref{kk2} and \ref{fin} where we 
discuss restrictions on homotopical height resulting from: \\
a) Lefschetz Hyperplane Theorem (Section \ref{lhpt}),\\
b) Formality (Section \ref{formality}), and\\
c) Group Cohomology (Section \ref{kk2}).\\

In Proposition \ref{cx=htstein}, we obtain the following exact computation of homotopical dimension in the class $\SSS$:
\begin{prop} If  $ht_\SSS (G)=n < \infty$, then there exists a smooth complex projective manifold $M$ of complex dimension exactly
$n$ realizing it.  Further, no smooth 
projective manifold $M$ of complex dimension less than $n$ realizes
$ht_\SSS (G)$. Hence $hdim_\SSS(G) = 2n$.

Furthermore, the universal cover $\til M$ of the above complex projective manifold $M$
is homotopy equivalent to a wedge of $n$--spheres.
\end{prop}

We also obtain partial positive answers to Question \ref{tol-h2}
(the details are in Propositions \ref{h2stein}, \ref{h2stein1}, \ref{hnstein}).

\begin{theorem}
If $ht_\SSS(G) > 2$, then either $H^2(G,\,\ZG)\,=\,0$ or $G$ is the 
fundamental group of a compact Riemann surface.

Suppose $ht_\SSS(G) \,=\, 2$, and $M$ is a smooth projective surface realizing 
the homotopical height of $G$. If in addition the action of $\pi_1(M)\,=
\,G$ on $\pi_2(M)$ is trivial, then either $H^2(G,\,\ZG)\,=\,0$ 
or $G$ is virtually the fundamental group of a compact Riemann surface.

If $ht_\SSS(G)\, >\, n$, then either $H^n(G,\,\ZG)\,=\,0$ or $G$ is
a {\rm PD}(n) group, in which case $H^n(G,\,\ZG)\,=\,\Z$. 
\end{theorem}

In Section \ref{kollar}, we note   a simple but curious connection  between
Koll\'ar's Question \ref{kollar-qn} and Carlson-Toledo's Conjecture \ref{tolcon}.

Section \ref{fin} is devoted to examining, generalizing and topologically characterizing certain counterexamples to Koll\'ar's Question \ref{kollar-qn}
discovered by Dimca, Papadima and Suciu.
In particular we prove the following Theorem (see Theorem 
\ref{dps-gen-th}) which provides (via Corollary \ref{dps-gen-cor}) a generalization of the counterexamples 
constructed in  \cite{dps-bb}. Our approach replaces the constructions using characteristic varieties  in \cite{dps-bb}
by some group cohomology computations.

\begin{theorem}
Let $f: M \longrightarrow S$ be an irrational topological Lefschetz fibration
that is not a Kodaira fibration, with $\dim M \,=\, 2n+2$, $n \geq 2$. 
Let $K$ be the finite critical set of $f$. Further suppose that $\til{M}$
is contractible.  Let $F$ denote the regular fiber and $N=\pi_1(F)$.  
   Then \\
a) $\pi_k(F) =0$ for $1<k<n$, \\
b) $\pi_{n}(F)$ is a free $\Z N$--module, 
with generators in one-to-one correspondence with $K\times \pi_1(S)$, \\
c) $\widetilde{F}$ is homotopy equivalent to a wedge of $n$--spheres, \\
d) $N$ cannot be of type FP; in particular, there does not exist a
quasiprojective $K(N,1)$ space. \end{theorem}

The examples in  \cite{dps-bb} come from complex analytic maps
$f: M\longrightarrow S$ between a product $M=S_1 \times \cdots \times S_r$ of Riemann surfaces and a Riemann surface $S$ of genus one.
The next Theorem (see Theorem \ref{charzn-th}) gives a purely topological characterization of homotopy classes of maps
$f: M\longrightarrow S$ that come from  irrational Lefschetz fibrations (that are not Kodaira fibrations). 

\begin{theorem}
Let $M$ be a product $S_1 \times \cdots \times S_r$ of closed, orientable
surfaces $S_i$ of positive genus with $r\,\geq\, 2$.
Fix $(x_1, \cdots , x_r )\,\in\, S_1 \times \cdots \times S_r$
and let $j_k: S_k \longrightarrow S_1 \times \cdots \times S_r$ be
given by $j_k(y_k) = (x_1, \cdots, x_{k-1}, y_k, x_{k+1},  \cdots, x_{r})$.

If there exists a product complex structure on $M$ (coming from complex structures on $S_k$) and an irrational Lefschetz fibration
 $f: M \longrightarrow S$ which is not a Kodaira fibration, then \\
a) $S$ has genus one, and\\
b) for all $k$, we have ${\rm deg} ((f\circ j_k))> [\pi_1(N): f_\ast(\pi_1(M))]$.\\

Conversely, suppose that $S$ is a closed orientable surface of positive genus and there exists
a continuous map $f: M \longrightarrow S$ such that  for all $k$,
$${\rm deg} ((f\circ j_k))\,>\, [\pi_1(N): f_\ast(\pi_1(M))]\, .$$
Then  $S$ has genus one and there exist complex structures on $S_k, S$
such that $f: M \longrightarrow S$ is homotopic to a Lefschetz fibration (for $M$ equipped with the product complex structure). 
Further,  there exist  complex analytic maps $\phi_k : S_k \longrightarrow S$ such that $\phi (z_1, \cdots , z_k) = m(\phi_1(z_1), \cdots , \phi_r(z_r))$,
where $m(.,.)$ denotes  addition on $S$. \end{theorem}

Finally, in Section \ref{zoo}, we end with a zoo of examples, counterexamples and questions. 

\section{Preliminaries}

\subsection{Preliminaries on CW complexes and smooth manifolds}

A well-known theorem going back to Dehn \cite{stillwell-bk}, Markov \cite{markov} and Boone-Haken-Poenaru \cite{bhp} states that
every finitely presented group appears as the fundamental group of a smooth closed 4-manifold. The main ideas of their proof can be generalized
as follows.

\begin{theorem}
The inequality $ht_\SM(G)\,\geq\, k$ holds for all finitely presented
groups $G$ of type $FP_k$. In particular, $ht_\SM(G)\,=\, \infty$
for all finitely presented groups $G$ of type $FP_\infty$.\label{sm-soft}
\end{theorem}

\begin{proof} Let $X$ be a $K(G,1)$ space with finitely many cells in each
dimension less than $(k+1)$, and let $X_k$ denote its $k$--skeleton. Embed $X_k$ simplicially
in a simplex $\Delta^N$ of large enough dimension $N$. Embed $\Delta^N$ simplicially in $\R^N$ where the latter is equipped with a standard
simplicial structure. Then $X_k$ is a simplicial subcomplex of $\R^N$ and
hence by \cite[p. 95, Ex. 3]{dold-bk}, the subspace $X_k$ is a strong deformation retract of a neighborhood
$U_k$ in $\R^N$. 

The argument in  \cite{bhp} (see also \cite[Ch 9.4]{stillwell-bk}) ensures that the boundary of $U_k$ in $\R^N$ may be taken to be a closed $(N-1)$ manifold
and hence the closure $W_k\,=\,\overline{U_k}$ is a manifold with boundary.
Let $\partial\, W_k \,= \,V_k$. Since $W_k$ has the same homotopy type as
$X_k$, it follows that $\pi_i(W_k)\,= \,0$ for $1\,<\,i\,<\,k$.

We now embed $\R^N$ in $\R^{N+k+1}$ and consider a regular neighborhood 
$Z_k$ of $W_k$ in $\R^{N+k+1}$. Then $$Y_k \,=\, \partial\, Z_k \,= \,(W_k
\times S^k) \cup (D^{k+1} \times V_k)\, ,$$ where $S^k$ and $D^{k+1}$ denote the
$k$--sphere and the $(k+1)$--disk respectively. Observe that $\pi_i(W_k
\times S^k)\,=\, 0$ for $1\,<\,i\,<\,k$, as $\pi_i(W_k) \,=\, 0$ for $1\,
<\,i\,<\,k$. Finally, consider the relative homotopy exact sequence of the pair
$(Y_k\, , W_k \times S^k)$:
$$
\longrightarrow \, \pi_i( W_k \times S^k)\,\longrightarrow \,
\pi_i( Y_k)\,\longrightarrow\, \pi_i(Y_k\, , W_k \times S^k)\,\longrightarrow\,
 \pi_{i-1}( W_k \times S^k)\, .$$
{}From it we have $\pi_i(Y_k)\,=\, \pi_i(Y_k\, , W_k \times S^k)$ for
$2\,<\,i\,<\,k$. Further, $$\pi_1 ( W_k \times S^k)\,=\, \pi_1(X_k)
\,= \,\pi_1(Y_k)$$ as in \cite{bhp}
(see also \cite[Ch 9.4]{stillwell-bk}), where the isomorphisms are natural and given via inclusions into $Z_k$.
Hence $\pi_i( Y_k)\,=\, \pi_i(Y_k\, , W_k \times S^k)$ for $2\,\leq\, i\,<\,k$.

Now, $\pi_i(Y_k\, , W_k \times S^k)\,=\, \pi_i ((D^{k+1} \times V_k)/\sim)$, where
$\sim$ identifies the boundary $S^k  \times V_k\,=\, \partial\,(D^{k+1} \times V_k)$ to a point.
Hence $\pi_i(Y_k\, , W_k \times S^k)\,=\,0$ for $i\,<\,k$
and so $\pi_i(Y_k) \,=\,0$ for $2\,\leq\, i\,<\,k$. Since $G\,=\,\pi_1(X_k)
\,=\, \pi_1(Y_k)$, the result follows. 
\end{proof}

Let $STT$ denote the class of smooth closed manifolds with stably trivial
tangent bundle. The proof of Theorem \ref{sm-soft} gives us more.
Observe that $Y_k$ is a codimension one closed submanifold
of $\R^{N+k+1}$. So the normal bundle to $Y_k$ in $\R^{N+k+1}$ is a trivial
line bundle. Hence the tangent bundle to  $Y_k$ is stably trivial. This
gives us the following corollary of Theorem \ref{sm-soft} which will be useful later:

\begin{cor}
The inequality $ht_{STT}(G)\,\geq\, k$ holds for all finitely presented groups
$G$ of type $FP_k$. In particular, $ht_{STT}(G)\,=\,\infty$ for all finitely presented groups $G$ of type $FP_\infty$.   \label{stt-soft} \end{cor}

\subsection{Preliminaries on Duality and Poincar\'e Duality groups}
We refer the reader to \cite{be, brown} for details about duality and
Poincar\'e Duality ({\rm PD}) groups.

\begin{lemma} Let $G$ be a one-ended group. Let $M$ be a closed orientable
$4$--manifold such that $\pi_1(M)\,=\,G$ and $\pi_2(M)\,=\,0$.
Then $M$ is a $K(G,1)$ space and $G$ is a {\rm PD}(4) group. \label{gurjar}
\end{lemma}

\begin{proof}
Let $\til{M}$ be the universal cover of $M$.
Since $\pi_2(M)\,=\,\pi_2(\til{M})$, it follows that $\pi_2(\til{M})\,=\,0$. 
As $\dim H^1(G,\,\ZG)+1$ coincides with the number of ends of $M$,
and $G$ is one-ended, it follows that $H^1(M,\, \ZG) = 0$. Hence
$H^1_c(\til{M})\,=\,0$. By Poincar\'e Duality, $H_3(\til{M})=0$.
Therefore, $\pi_3(\til{M})\,=\,0$ by Hurewicz' Theorem. 

Also, since $G$ is one-ended, it is infinite. Hence $\til M$ is non-compact,
and $H_4(\til{M})\,=\,0$. Again, by Hurewicz' Theorem, $\pi_4(\til{M})\,=\,0$. 
Further, since $H_i(\til{M})\,=\,0$ for all $i\,>\,4$, we conclude that
$\pi_i(\til{M})\,=\,0$ for all $i\, \geq\, 1$.
Consequently, $\til M$ is contractible and the result follows.
\end{proof}

We shall also need the following theorem due to  Kleiner (see also \cite{bowditch}).

\begin{theorem}[\cite{kleiner}]\label{kleiner}
Suppose 
$G$ is a group which is ${\rm FP}_2$ over $\Z$. If $H^2 (G,\,\ZG)$ has a 
1--dimensional $G$--invariant submodule, then $G$ is virtually a surface group. 
Hence $H$ is a ${\rm PD}(2)$ group over $\Q$ if and only if $H$ is virtually 
a surface group.
\end{theorem}

\subsection{Preliminaries on K\"ahler Groups} We begin with some known facts and simple observations. 

We recall a well-known theorem of Gromov.

\begin{theorem}[\cite{gro-free}]\label{1end}
Any infinite K\"ahler group is one-ended.
\end{theorem}

A more general result is proved in \cite{dg}.

Lemma \ref{gurjar} and Theorem \ref{1end} together yield the following:

\begin{cor}\label{surf} 
Let $M$ be a smooth complex projective surface with $\pi_2(M)\,=\,0$.
Then $M$ is a $K(\pi_1(M),1)$ space, and $\pi_1(M)$ is a {\rm PD}(4) group.
\end{cor}

A similar statement holds in every dimension. Towards this. we
 first recall a classical theorem of Andreotti, Frankel and Narasimhan \cite{af, an, nar}.

\begin{theorem}\label{af} 
Let $M$ be a (not necessarily smooth) complex projective variety of
complex dimension $n$ such that the universal
cover $\til M$ is Stein. Then $H_i(\til{M},\, {\mathbb Z})\,=\, 0$ for $i\,>\,n$. In
fact $\til{M}$ is homotopy equivalent to a CW--complex of real dimension
$n$. Also, the second homotopy group $\pi_2(M)$ is free abelian.

If $X$ is any Stein manifold embedded in $\cc^N$ for some $N$, then for a dense set of points $x\,\in\,
\cc^N$, the square of the distance function
$z\,\longmapsto\, d^2(x\, , z)$ is Morse and strongly plurisubharmonic on $X$.
\end{theorem}

\begin{prop}\label{ndim} 
Let $M$ be a  smooth complex projective variety of complex dimension $n$ 
whose universal cover $\til M$ is Stein. 
Also suppose that $\pi_i(M)\,=\,0$ for all $1\,<\,i\,\leq\, n$.
Then $M$ is a $K(\pi_1(M),1)$ space, and $\pi_1(M)$ is a {\rm PD}(2n) group.
In particular, $\pi_i(M)\,=\,0$ for all $i\,>\,1$.
\end{prop}

\begin{proof} Since $\pi_i(M)\,=\, \pi_i(\til{M})$ for $i\, >\, 1$,
it follows that $\pi_i(\til{M})\,=\,0$ for $1\,<\,i\,\leq\, n$. From Theorem \ref{af}
and repeated application of Hurewicz' Theorem, it follows that $\pi_i(\til{M})
\,=\,0$ for all $i$, and the proposition follows.
\end{proof}

\begin{rmk} For the conclusion in Proposition \ref{ndim} we do not quite need to assume that universal cover $\til M$ is Stein.
It suffices to assume that $\pi_i(M)\,=\,0$ for $1\,<\,i \,<\, 2n-1$ in
view of the argument in the proof of Lemma \ref{gurjar}. 
\end{rmk}

\subsection{The trivial group}\label{sec1.6}

We illustrate the notion of homotopical height with the example of the 
trivial group, and justify the dimensional restriction in Definition 
\ref{htd}.

\begin{prop} For the trivial group $G\,=\,\{ 1 \}$,
we have $ht_\SP (G)\,=\, ht_\KK (G)=ht_\PP (G) = ht_\HC (G) =ht_\SSS (G)=2$,
and $ht_\CT  (G) =   ht_\CR  (G)= ht_\St(G)=ht_\aff(G)=ht_\CA(G)=ht_\AC(G)=\infty$.
\label{trivial} \end{prop}

\begin{proof} Let $M$ be a simply connected 
closed symplectic manifold. Then, by the Hurewicz' Theorem, $H_2(M)\,=\,\pi_2(M)$. Since $M$ is symplectic, $H_2(M)\,\neq\, 0$
and it follows that $ht_\SP (G)\,=\,2$. As $\KK$ is a subclass of
$\SP$, and the 2-sphere $S^2$ belongs to the class, it follows that
$ht_\SP (G)\,= \,ht_\KK (G)\,=\,2$. Also, since we are considering only
manifolds of positive dimension,
the classes $\PP, \HC, \SSS$ are all subclasses of $\KK$ and so $ht_\PP (G)\,=\,ht_\HC (G)\,=\,ht_\SSS (G)=2$.

Next, since the sphere of dimension $(2n+1)$ admits both a contact and CR structure, it
follows that  $ht_\CT  (G)\,=\, ht_\CR  (G)\,=\,\infty$.

Since the affine plane is Stein, we have $ht_\St  (G)\,=\,ht_\aff(G)\,=\, \infty$.

Finally, the Calabi--Eckmann manifolds are complex with underlying real
manifold $S^{2m+1} \times S^{2n+1}$. Therefore, we have 
$ht_\CA(G)\,=\, \infty$. Complex manifolds are, of course, almost complex,
implying that $ht_\AC(G)\,=\, \infty$. 
\end{proof}

\section{Soft Classes of Manifolds}\label{soft}

\subsection{Open almost complex manifolds}

We first show that the category $ACO$ of open almost complex manifolds is 
soft of Type 1.

\begin{theorem}
The inequality $ht_{ACO}(G)\,\geq\, k$ holds for all finitely presented groups
$G$ of type $FP_k$. In particular, $ht_{ACO}(G) = \infty$  for all finitely presented groups $G$ of type $FP_\infty$. \label{aco-soft} \end{theorem}

\begin{proof}
By Corollary \ref{stt-soft}, we have $ht_{STT}(G)\,\geq\, k$ for all finitely
presented groups $G$ of type $FP_k$. In particular, $ht_{STT}(G)\,=\,\infty$ for
all finitely presented groups $G$ of type $FP_\infty$. Hence for a finitely
presented group $G$ of type $FP_k$, there exists a manifold $M$ such that
\begin{enumerate}
\item[(a)] $\pi_1(M) \,= \,G$,
\item[(b)] $\pi_i(M)\,=\,0$ if $1<i<k$, and
\item[(c)] the tangent bundle $TM$ of $M$ is stably trivial.
\end{enumerate}

Let $E_m$ denote the trivial vector bundle of rank $m$ over $M$. Then there exists
an integer $N$ such that the vector bundle $W_M \,= \,TM \oplus E_m$ over $M$ is
trivial for all $m \,\geq\, N$.
Taking $m\,\geq\, N$ to be odd or even according as the dimension of $M$ is
odd or even, we see that the open manifold $M \times \R^m$ is even
dimensional and has a trivial
tangent bundle. So, $T(M \times \R^m) \,=\,
(M \times \R^m)\times \R^{2k}\,= \,(M \times \R^m)\times \cc^{k}$.
Hence $M \times \R^m$ is almost complex. The theorem follows.
\end{proof}

\begin{rmk} \label{ac2stein} Note that in the proof of Theorem \ref{aco-soft}, $m$ (the dimension of the factor $\R^m$) can be taken to be arbitrarily large.
This will be useful in the proof of Theorem \ref{stein-soft} below.
\end{rmk}

\subsection{Almost CR and almost complex manifolds}

Following \cite{blair-bk}, we use
the terminology {\it almost CR} for manifolds whose tangent bundles have a
codimension one distribution equipped with a complex structure $J$. If,
moreover, this distribution is formally integrable, we say that the
manifold is CR. Sometimes, in the literature, the former is called CR and
the latter integrable CR. We explicate this to avoid confusion.
We denote the class of closed  almost CR manifolds by $ACR$.

We observe that Theorem \ref{aco-soft} immediately furnishes softness 
(Type 1) for closed almost CR  manifolds 
and almost complex manifolds. Theorem \ref{ac-soft} below is a stronger 
version of a Theorem of Kotschick \cite{kotschick-cont} that every finitely presented group occurs as the fundamental group of an almost contact manifold.

\begin{theorem}
The inequality $ht_{ACR}(G)\,\geq\, k$ holds for all finitely presented 
groups $G$ of type $FP_k$. In particular, $ht_{ACR}(G)\,=\,\infty$ for all 
finitely presented groups $G$ of type $FP_\infty$.

The inequality  $ht_\AC(G)\,\geq\, k$ holds for all finitely presented groups
$G$ of type $FP_k$. In particular, $ht_\AC(G)\,=\,\infty$ for all finitely
presented groups $G$ of type  $FP_\infty$. \label{ac-soft}
\end{theorem}

\begin{proof}
Given a finitely presented group $G$ of type $FP_k$, by  Theorem \ref{aco-soft}
there exists a manifold $M$ and a natural number $n$ such that
\begin{enumerate}
\item[(a)] $\pi_1(M) \,=\, G$,
\item[(b)] $\pi_i(M)\,=\,0$ for all $1\,<\,i\,<\,k$, and
\item[(c)] $M\times\R^m$ carries an almost complex structure
$(M \times \R^m, J)$ for all $m \geq n$ such that $m+\dim M$ is even.
\end{enumerate}

The smooth real hypersurface $H\,=\, M \times S^{m-1}$ of the 
almost complex manifold $(M \times \R^m\, , J)$ has a 
complex tangent subspace $T H \cap J(T H) \subset TH$   forming a (not necessarily integrable) complex tangent distribution. Clearly,  $T H \cap J(T H)$
is of real
codimension one in $TH$. Hence $J$ restricted to $T H$ gives an almost CR-structure on $H$.
If $m > k$, then
\begin{enumerate}
\item[(a)] $\pi_1(H) \,=\, G$, and

\item[(b)] $\pi_i(H)\,=\, 0$ for $1\,<\,i\,<\,k$.
\end{enumerate}
The first statement of the theorem, which is for closed almost CR manifolds,
follows.

To prove the second statement, consider $W\,=\,H \times S^{2m+1}$. Both $H$
and $S^{2m+1}$ admit almost CR structures; let
$\HH\,\subset\, TH$ and $\DD\,\subset\, TS^{2m+1}$ be distributions
giving almost CR structures on $H$ and $S^{2m+1}$
respectively. Let $\xi$ and
$\eta$ denote the unit vector fields on $H$ and $S^{2m+1}$ normal to $\HH$ and
$\DD$ respectively. Let $J_H$ and $J_S$ denote the almost complex structures on
$\HH$ and $\DD$ respectively.

We define an almost complex structure $J$ on $TW$ as follows.
First, identify $TW$  naturally with the direct sum  $\HH \oplus \DD \oplus \R
\xi \oplus \R \eta$. Now define, for every $p \in W$,
a linear homomorphism $J_p\,:\,T_pW\,\longrightarrow\, T_pW$ by
\begin{enumerate}
\item $J_p(v)\,=\, J_H(v)$ for $v \in \HH$,
\item $J_p(v)\,=\, J_D(v)$ for $v \in \DD$,
\item $J_p(\xi_p)\,=\, \eta_p$, and
\item $J_p(\eta_p)\,=\, -\xi_p$.
\end{enumerate}
Since $J_p$ is clearly a smoothly varying family of automorphisms satisfying
$J_p^2\,=\, - {\rm Id}$, it is an almost complex structure on $W$. Further, 
\begin{enumerate}
\item[(a)] $\pi_1(W) \,=\, G$, and

\item[(b)] $\pi_i(W)\,=\,0$ for $1\,<\,i\,<\,k$.
\end{enumerate}
The second statement of the theorem now follows.
\end{proof}

\subsection{Stein, CR and Contact manifolds}

We shall require the following celebrated theorem of Eliashberg to move from
the category of open almost complex manifolds to the category of Stein manifolds.

\begin{theorem}[\cite{eli-stein}]\label{eli-stein}
Let $X$ be a real $2n$--dimensional smooth manifold, with $n\,>\,2$, such that
$X$ carries an almost complex structure $J$ and a proper Morse function $\phi$ 
all of whose critical points are of index at most $n$. Then $J$ is homotopic
to a complex structure $J^\prime$ on $X$ with respect to which $\phi$ is
plurisubharmonic. In particular, $(X\, ,J^\prime)$ is Stein.
\end{theorem}

We now have the tools ready to show that the category $\St$ of Stein 
manifolds is soft of type 1.

\begin{theorem} 
The inequality $ht_{\St}(G)\,\geq\, k$ holds for all finitely presented groups
$G$ of type $FP_k$. In particular, $ht_{\St}(G)\,= \,\infty$  for all finitely
presented groups $G$ of type $FP_\infty$. \label{stein-soft}
\end{theorem}

\begin{proof}
By Theorem \ref{aco-soft}, we have $ht_{ACO}(G)\,\geq\, k$ for all finitely
presented groups $G$ of type $FP_k$. In particular, $ht_{ACO}(G)\,=\, \infty$
for all finitely presented groups $G$ of type $FP_\infty$. By Remark
\ref{ac2stein}, for a finitely presented group $G$ of type $FP_k$, there exists an
open almost complex manifold $(M\times \R^m\, , J)$ with $M$ compact such that
\begin{enumerate}
\item[(a)] $\pi_1(M\times \R^m)\,=\,\pi_1(M)\,=\,G$,
\item[(b)] $\pi_i(M\times \R^m)\,=\,\pi_i(M)\,=\,0$ if $1\,<\,i\,<\,k$, and
\item[(c)] $m \,\geq\, \dim M$.
\end{enumerate}

Since $m\,\geq\,\dim M$, there exists  a proper Morse function $\phi$ on
$M\times \R^m$ all of whose critical points are of index at most
$\frac{1}{2} \dim (M\times \R^m)$. Hence by Theorem \ref{eli-stein}, the almost
complex structure $J$ is homotopic to a complex structure $J^\prime$ such
that $(M\times \R^m\, , J^\prime)$ is Stein. The theorem follows.
\end{proof}

As an immediate corollary of Theorem \ref{stein-soft}, we obtain softness 
for the class of closed integrable CR manifolds as well as closed 
contact manifolds.

\begin{theorem}
The inequality $ht_{\CR}(G)\,\geq\, k$ holds for all finitely presented groups
$G$ of type $FP_k$. In particular, $ht_{\CR}(G)\,=\,\infty$  for all finitely presented groups $G$ of type $FP_\infty$.

The inequality $ht_\CT(G)\,\geq\, k$ holds for all finitely presented groups
$G$ of type $FP_k$. In particular, $ht_\CT(G)\,=\,\infty$ for all finitely
presented groups $G$ of type $FP_\infty$. \label{ct-soft}
\end{theorem}

\begin{proof}
Let $G$ be of type $FP_k$. Then by Theorem \ref{stein-soft}, there exists a
Stein manifold $M$ such that
\begin{enumerate}
\item $\pi_1(M)\,=\, G$, and
\item $\pi_i(M)\,=\, 0$ for all $1\,<\,i\,<\,k$.
\end{enumerate}
By the property of being Stein, there exists a natural number $N$ such that
$M$ can be embedded in $\cc^N$. We assume that $M$ comes equipped with such an embedding.
The second part of Theorem \ref{af} says that the square of the distance
function $z\,\longmapsto\, d^2(x\, ,z)$ is Morse and strongly plurisubharmonic for a dense
set of points $x\, \in\, \cc^N$. Without loss of generality, let the origin $0$ be such a point.
Then the intersection with the sphere $S(0\, ,R) \cap M \,=:\,M_R $ is a
strongly pseudo-convex integrable CR manifold for all $R$ that
are regular values of the function $z\,\longmapsto\, d^2(0\, ,z)$.

Since $G$ is of type $FP_k$, the above Morse function has finitely many critical values of index less than $k$. Suppose all critical values of index less than $k$
are contained in the ball of radius $R_0$ about $0$. Then for all $R>R_0$, the
strongly  pseudo-convex integrable CR manifold $M_R$ satisfies the following
two conditions:
\begin{enumerate}
\item[(1)] $\pi_1(M_R)\,=\, G$, and

\item[(2)] $\pi_i(M_R)\,=\,0$ for all $1\,<\,i\,<\,k$.
\end{enumerate}
Therefore, he softness of closed integrable CR manifolds (the first statement
of the theorem) follows.

Since strongly  pseudo-convex integrable closed CR manifolds are contact \cite[p. 72-73]{blair-bk}, 
softness of closed contact manifolds (the second statement of the theorem)
follows.
\end{proof}

A'Campo and Kotschick, \cite{acko}, showed that an arbitrary finitely presented group can be realized
as the fundamental group of a closed contact manifold of 
dimension five. Consequently, closed contact manifolds are of Type 3 in our classification. Theorem \ref{ct-soft}
upgrades $\CT$ from Type 3 to Type 1 in the softness scale.

\section{Complex and Symplectic Manifolds} \label{casp}

The classes $\SP$ and $\CA$ of closed symplectic manifolds and closed complex 
manifolds respectively, exhibit intermediate behavior in terms of 
softness and hardness. By Theorem \ref{gompf-pi1} and Theorem
\ref{taubes-pi1}, both $\SP$ and $\CA$ are of Type 3.

\subsection{Complex Manifolds}

The class $\CA$ is also of Type 2.
However, we do not know whether the class $\CA$ can be upgraded to Type 1.
We rephrase and elaborate Question \ref{yau-qn-ht}, which essentially asks
if this can be done.

\begin{qn} \label{yau-qn-ht2}
Is $ht_\CA (G) \,\geq\, k$ for every finitely presented group $G$ of type
$FP_k$? For a finitely presented group $G$ of
type $FP_\infty$, is $ht_\CA (G) \,=\, \infty$?
\end{qn}

\begin{rmk} \label{ac2ca} {\rm
Theorem \ref{ac-soft} shows that the answer to Question \ref{yau-qn-ht2}
is ``Yes'' if the class $\CA$ is replaced by the class $\AC$ of closed almost complex manifolds.
Hence a group for which the answer to Question
\ref{yau-qn-ht2} is negative will produce an almost complex manifold
which does not admit a complex structure.}
\end{rmk}

Proposition \ref{trivial} answers Question \ref{yau-qn-ht2} positively for the trivial group. Note that Taubes' twistor construction
in \cite{taubes} furnishes manifolds whose second homotopy group is non-zero and hence we cannot hope to answer Question \ref{yau-qn-ht2} 
positively using his techniques.
We give a number of other examples of $G$ below for which the answer of
Question \ref{yau-qn-ht2} is positive.

We shall need the following theorem due to Morimoto. Let $M$ be an almost contact manifold. The almost contact structure
induces a natural almost complex structure on $M \times \R$. If this almost
complex structure is  integrable, then the almost contact manifold $M$ is said
to be {\it normal}. In particular, a Sasakian manifold admits a normal almost contact structure.

\begin{theorem}[\cite{morimoto-nacs}]\label{morimoto}
The Cartesian product of two  normal almost contact manifolds admits a  complex structure.
\end{theorem}

\begin{lemma}
The height $ht_\CA (G) \,=\, \infty$ if $G$ is any of the following groups:
\begin{enumerate}
\item[(a)] finitely generated abelian groups, and

\item[(a)] fundamental group of any spherical space form.
\end{enumerate}
\label{ht-ca-eg}
\end{lemma}

\begin{proof}
Groups in (a):\, By taking Cartesian products of manifolds, it reduces
to showing that $ht_\CA (G) \,= \,\infty$ for every cyclic group $G$.
Now, if $G \,=\, \Z$, then the Hopf manifolds $M_n$ homeomorphic to
$S^1 \times S^{2n+1}$ have $\pi_i(M_n) \,=\, 0$ for $1\,<\, i\,\leq\, 2n$.
Hence $ht_\CA (\Z) \,=\, \infty$.

For any positive integer $p$, let $L_p^{2n+1}$ denotes the
$(2n+1)$--dimensional Lens space with fundamental group $\Z/p\Z$. Then 
$W_n \,=\,L_p^{2n+1} \times S^{2n+1}$ admits a complex structure by
Theorem \ref{morimoto}, and also $\pi_i(W_n)\,=\,0$ for $1\,<\,i\,\leq\, 2n$.
Hence $ht_\CA (\Z_p)\,=\, \infty$.

\medskip
Groups in (b):\, Let $M \,=\, S^N/G$, where $G$ acts freely by isometries on
the round sphere $S^N$ of odd dimension $N$. Take $m$ times iterated join
$S^N \ast S^N \ast \cdots \ast S^N$. The group $G$ acts on it freely using
the diagonal action, with $G$ acting on each factor as before. Therefore,
we obtain a free action by isometries
of $G$ on the round sphere $S^{mN+m-1}$. Define $M_m\,:=\,S^{mN+m-1}/G$, and
let $W_m \,:=\, M_m \times S^{2m+1}$. Then, by Theorem \ref{morimoto}
again, $W_m$ admits a complex structure, and $\pi_i(W_m)\,=\, 0$ for
$1\,<\, i\,\leq\, 2m$. Hence $ht_\CA (G)\,=\, \infty$.
For a classification of fundamental groups of spherical space
forms, see \cite[Ch.~7]{wolf-bk}.\end{proof}

We shall now rework an argument going back to Oliver \cite[p. 
546--7]{oliver} to provide a ``very inefficient'' model space for arbitrary 
finite groups. The models in Lemma \ref{ht-ca-eg} are considerably more 
efficient. In fact we shall be using the argument for finite cyclic groups 
in the proof of Lemma \ref{ht-ca-eg} implicitly in the proof of 
Proposition \ref{ht-ca-finite} below.

\begin{prop}
For any  finite group $G$,
$$ht_\CA (G) \,=\, \infty\, .$$
\label{ht-ca-finite} \end{prop}

\begin{proof}
Let $ G$ be any finite group. Let $ H$ be a  subgroup of $G$ and $M$ a complex manifold on which $H$ acts freely
by holomorphic automorphisms. Denote the index $\# G/H$ by $N$. We identify the
product complex manifold $M^N$ with the space of all maps from $G/H$ to $M$.
Then $M^N$ can be naturally identified with the space of 
all $H$--equivariant maps from $G$ to $M$. The action of $H$ on $M^N$ is
the diagonal one. This action of $H$ on $M^N$ naturally extends to
a $G$--action on $M^N$ using the left--translation action of $G$
on $G/H$. Clearly this action of $G$ on $M^N$ is also by
holomorphic automorphisms.

We denote by $C_n$ the Calabi--Eckmann manifold with underlying manifold
$S^{2n+1} \times S^{2n+1}$. We observe that for any $m$,
the cyclic group $\Z/m\Z$ acts freely by complex 
automorphisms on $C_n$ (see \cite{morimoto-nacs} for instance). Hence for any non-trivial
element $g\,\in\, G$, the cyclic group $\langle g\rangle \,\subset\, G$ acts freely by 
complex automorphisms on $C_n$.

Let $i_g\,=\, [G : \langle g\rangle]$ denote the index of $\langle g\rangle$ in $G$. We
have seen in the first part of the proof that there is an action of $G$ on
the Cartesian product $C_n^{i_g}$ by
holomorphic automorphisms, such that the action of the subgroup $\langle g\rangle$ is 
free. Hence the diagonal action of $G$ on the product $$W_n \,:=\, \prod_{g\in G, 
g\neq 1} C_n^{i_g}$$ is free and by holomorphic automorphisms. Let $V_n \,=\, 
W_n/G$ denote the quotient manifold. Then $\pi_1(V_n) \,=\, G$, and 
$\pi_i(V_n)\,=\,0$ for $1\,<\,i\,\leq\, 2n$.

Hence $ht_\CA (G) \,=\, \infty$ for any finite group $G$. \end{proof}

\begin{prop}\label{ht-ca-nacs}
The height $ht_\CA (G) \,=\, \infty$ if $G$ is one of the following:
\begin{enumerate}
\item[(a)] the (integral) Heisenberg group $H^{2m+1}$, and

\item[(b)] the fundamental group of an aspherical closed Sasakian manifold, or, more
generally, an aspherical  manifold equipped with an almost normal contact structure.
\end{enumerate}
\end{prop}

\begin{proof}
Groups in (a):\, In this case $G$ is the fundamental group of a non-trivial circle
bundle over an even dimensional torus. Hence $G\,=\,\pi_1(M)$, where
$M$ admits a normal almost contact structure. By Theorem \ref{morimoto}, the manifold
$V_n\,:=\,M \times S^{2n+1}$ admits a complex structure.
We have $\pi_1(V_n) \,=\, G$ and $\pi_i(V_n)\,=\,0$ for $1\,<\,i\,\leq\, 2n$.
Hence $ht_\CA (G) \,=\, \infty$.

Groups in (b):\, The argument is same as in part (a), namely take products with odd
dimensional spheres.
\end{proof}

\begin{rmk} \label{ht-sas} If {\it NACS} (respectively, $SAS$) denotes the category of closed manifolds equipped with an almost normal contact structure
(respectively, closed Sasakian manifolds), then the above argument shows that
$$ht_{SAS} (G)\,\leq\, ht_{\it NACS}(G)\,\leq\, ht_\CA(G)$$ for all $G$.
\end{rmk}

The next question gives two natural test cases for Question \ref{yau-qn-ht2}.

\begin{qn} \label{yau-qn-ht3} Is $ht_\CA (G) \,=\, \infty$ for the 
solvable groups or the fundamental groups of closed real hyperbolic manifolds 
of dimension greater than two?
\end{qn}

\subsection{Symplectic Manifolds} All
 the results in this subsection apply to the more general class of c--symplectic
(cohomologically symplectic) manifolds, i.e., closed manifolds $M^{2n}$ with
$\omega\,\in\, H^2(M,\, {\mathbb R})$ such that $\omega^n\,\neq\, 0$.  We, nevertheless,
prefer to retain the notation $\SP$ and talk about symplectic manifolds
to prevent further proliferation of categories.

Theorem \ref{trivial} shows that $ht_\SP(\{ 1 \})\,=\, 2$. This phenomenon is far more general
as shown by the following proposition.

\begin{prop} Let $G$ be a finitely presented group such that $b_2(G)\,=\, 0$. Then
$ht_\SP(G)\,=\, 2$.  Similarly, if $b_{2k}(G)\,=\,0$, then $ht_\SP(G)\,\leq \,2k$.
\label{b20-sp}
\end{prop}

\begin{proof}
Suppose it is not true. Then there exists a closed symplectic manifold $M$ such that
$\pi_1(M)\,=\,G$ and $\pi_2(M)\,=\,0$.
By attaching cells of dimension greater than three to $M$, we obtain a $K(G,1)$
space which will be denoted by $X$. Then $b_2(M)\,=\, b_2(X)\,= \,b_2(G)\,=\,0$.
This contradicts the assumption that $M$ is symplectic. 

Similarly, if $b_{2k}(G)\,=\,0$, and $M$ is a closed symplectic manifold such that
$\pi_1(M)\,=\,G$ and $\pi_i(M)\,=\,0$ for $1\,<\,i\,\leq\, 2k$, then we can construct
a $K(G,1)$ space $X$ by attaching cells of dimension greater than $(2k+1)$ to $M$. 
In that case, $b_{2k}(M)\,=\, b_{2k}(X)\,=\, b_{2k}(G)\,=\,0$, contradicting the fact that
$b_{2k}(M) \,> \,0$ (ensured by the existence of the symplectic structure).\end{proof} 

It follows that all finite groups and the infinite cyclic group $G\,=\,\Z$ have 
$ht_\SP(\Z)\,=\, 2$, and, more generally, $b_{2k}(\Z^{2k-1} )\,\leq\, 2k$.

Recall that the $\CC$--homotopical dimension $hdim_\CC(G)$ is the smallest number among
dimensions of manifolds realizing $ht_\CC(G)$.
Lemma \ref{gurjar} furnishes the following lower bounds on $hdim_\SP$.

\begin{prop}\label{hdim-gur} If a one-ended finitely presented
group $G$ has $ht_\SP(G)\,>\, 2$,
then
\begin{enumerate}
\item[(a)] either $hdim_\SP(G)\,=\,2$, and $G$ is the fundamental group of a 2-manifold of
positive genus,

\item[(b)] or $hdim_\SP(G)\,=\,4$ and $G$ is a {\rm PD}(4) group,

\item[(c)]  or $hdim_\SP(G)\,\geq\, 6$.
\end{enumerate}
\end{prop}

\begin{proof} Suppose that $hdim_\SP(G) \,< \,6$. Then there exists a symplectic manifold $M$ of dimension 2 or 4 such that 
$\pi_1(M)\,=\, G$ and $\pi_2(M)\,=\,0$. If $M$ is 2-dimensional,
then we land in case (a).
If $M$ is 4-dimensional, then Lemma \ref{gurjar} says that $G$ is a PD(4) group.
\end{proof}

Fundamental groups of closed orientable surfaces of positive genus are called {\it surface groups}.

The next observation shows that $hdim$ imposes restrictions on homotopical height.
This might be an approach to Question \ref{tol-h2}.

\begin{prop}
Assume that $G$ is not a surface group, and $H^2(G,\, \ZG)\,\neq\, 0$. Also, assume
that $hdim_\SP(G)\,=\,4$. Then, $ht_\SP(G) \,= \,2$. \end{prop}

\begin{proof}
Suppose that $ht_\SP(G) \,>\, 2$. Consequently, there exists a symplectic 4-manifold $M$
with $\pi_1(M)\,=\, G$ and $\pi_2(M)\,=\,0$. Therefore, $\pi_2(\til{M})\,=\,0$, where
$\til M$ is the universal cover of $M$. By the Hurewicz' Theorem, $H_2(\til{M})\,=\,0$,
and hence by Poincar\'e duality, $H^2_c(\til{M})\,=\,0$. Therefore $H^2 (M, \,\ZG)
\,=\,0$. Since $\til{M}$ is 2-connected, and has a finite 3-skeleton modulo $G$,
we have $H^2 (M,\, \ZG)\,=\,H^2 (G,\, \ZG)$. Hence $H^2 (G, \ZG)\,=\,0$,
yielding a contradiction.
\end{proof}

The next proposition provides some weak evidence towards Conjecture \ref{tolcon}.

\begin{prop}\label{toledob2}
Let $G$ be an infinite finitely presented group
with $ht_\SP (G) \,>\, 2$. Then $b_2 (G)$ is positive.
\end{prop}

\begin{proof} As usual,
first suppose that $ht_\SP (G)$ is finite.
Let $M$ be a closed smooth symplectic manifold realizing $ht_\SP (G)$. Then $\pi_2(M)\,=
\,0$.  Hence a $K(G,1)$ space $X$ may be constructed from $M$ by adding cells in dimensions greater than three to kill higher homotopy groups.
Therefore, $b_2(G)\,=\, b_2(X)\,=\, b_2(M)$; the last equality holds because $M$ and $X$ differ
only in cells of dimension four or more.
Since $M$ is symplectic, we have $b_2(M) > 0$, and the proposition follows in this case.

Next, assume that $ht_\SP (G)\,=\, \infty$. Then there exists $M \,\in\, \SP$
with $\pi_1(M)\,=\, G$ and $\pi_2(M)\,= \,0$. The same argument now furnishes the conclusion that $ b_2 (G)$ is positive. \end{proof}

\begin{rmk}  \label{toledob2app} We 
note that Proposition \ref{toledob2} offers the following approach to Conjecture 
\ref{tolcon}. Given an infinite K\"ahler group $G$ and a compact K\"ahler manifold $M$ with 
fundamental group $G$, can one construct a new symplectic (or 
c--symplectic) manifold $M_1$ from $M$ by attaching cells in dimensions greater than 
three such that $\pi_2(M_1) \,=\, 0$? This construction can be generalized a bit by 
demanding that we are required to attach cells only to a finite-sheeted cover of $M$, 
converting the latter into a manifold with vanishing $\pi_2$.

For groups $G$ admitting such constructions, Proposition \ref{toledob2} shows that 
virtually $ b_2 (G)$ is positive. The advantage of the approach is that it takes the 
problem into the domain of symplectic topology, where surgery techniques are much more 
plentifully available. In particular, K\"ahler symplectically aspherical groups of class 
$\mathcal A$ in the terminology of \cite{ikrt} verify Conjecture 
\ref{tolcon}.
\end{rmk}

Cohomological dimension also imposes restrictions on homotopical height for 
straightforward reasons as follows.

\begin{prop}\label{cdhtsp}
If $cd(G)\,=\, 2n$, then either $ht_\SP(G) \,=\, \infty$ with $G$ being a {\rm PD}(2n)
group, or $ht_\SP(G) \,\leq\, 2n+2$. If $cd(G)\,=\, 2n+1$, then $ht_\SP(G)\,\leq\, 2n+2$.
\end{prop}

\begin{proof} Suppose that $G$ is not a PD(2n) group that is realizable as the fundamental
group of a symplectic $K(G,1)$ manifold.
Also, assume that $cd(G)\,=\, 2n$ or $2n+1$. If $ht_\SP(G) \,>\, 2n+2$, then there exists a symplectic manifold $M$ such that
$\pi_1(M)\,=\,G$ and $\pi_i(M) \,=\,0$ for $1\,<\,i\,\leq\, 2n+2$. By attaching cells of
dimension greater than $2n+3$ we obtain a $K(G,1)$ space $X$. Then
$$b_{2n+2}(G, \,\Z)\,=\, b_{2n+2}(X,\, \Z)\,=\,b_{2n+2}(M,\, \Z)\,>\,0\, ,$$
contradicting the assumption that $cd(G)\,=\, 2n$ or $2n+1$.
\end{proof}

We shall give {\it lower bounds} on $ht_\SP(G)$ for certain groups including $\Z^N$ (for
$N$ odd) later on (see Proposition \ref{csp-abelian})
using some techniques from the hard category of projective manifolds, in particular the
Lefschetz hyperplane theorem.

\subsection{Affine and Quasi-projective varieties}

Let $\aff$ and $\QP$ denote the classes of smooth affine and smooth quasiprojective varieties respectively.
Homotopical height does not distinguish between these categories:

\begin{lemma} For $G$ a finitely presented group, we have
$ht_\aff (G) = ht_\QP (G)$. \label{aff=qp} \end{lemma}

\begin{proof}
Jouanolou showed that for any quasi-projective variety $M$,
there exists an affine variety $W$  mapping surjectively onto $M$ with contractible
fibers \cite{jo}. In particular, this $W$ is homotopy equivalent to $M$. It follows that $ht_\aff(G) \geq ht_\QP (G)$.

Since $\aff \subset \QP$, we have $ht_\QP(G) \geq ht_\aff (G)$.\end{proof}

\section{Basic Restrictions on Homotopical Height of K\"ahler Groups}\label{kk1}

This and the next two sections deal with the ``hard'' categories of K\"ahler ($\KK$) and complex
projective ($\PP$) manifolds. We shall also deal with the subclasses $\HC$ and $\SSS$ of 
$\PP$ consisting of smooth projective varieties whose universal covers are, 
respectively, holomorphically convex and Stein. This section deals with some basic 
restrictions on homotopical height coming from the Lefschetz hyperplane theorem and 
formality.

\subsection{Lefschetz Hyperplane Theorem}\label{lhpt}

\begin{lemma}\label{cx=ht}
If $-\infty \, <\, ht_\PP (G)\, <\, \infty$, then there exists
a smooth complex projective variety $M$ realizing it such that the
complex dimension of $M$ is at most $ht_\PP (G)$.
\end{lemma}

\begin{proof}
Let $N$ be a smooth complex projective manifold realizing 
$ht_\PP (G)\,=\,n$. Assume that the complex dimension
of $N$ is strictly bigger than $n$ (otherwise there is nothing to prove).
Let $N_1$ be a hyperplane section of $N$. By the  Lefschetz hyperplane
theorem, we have $\pi_i(N_1)\,=\, \pi_i(N)$ for all
$i \,<\, \dim_{\mathbb{C}} N_1$. In particular,
$\pi_1(N_1) \,=\, G$ and $\pi_i(N_1) \,=\, 0$ for $1\,<\,i\,<\,n$. We can repeat
this argument replacing $N$ by $N_1$ and iterating till the dimension of the
hyperplane section reduces to $n$.
\end{proof}

\begin{rmk}\label{cx=htrmk}
The same argument works for the categories $\HC$ and $\SSS$ as hyperplane
sections preserve both the categories $\HC$ and $\SSS$ (see Proposition 2.1 of \cite{bm-cd}
for instance). Thus if $ht_\HC (G)\,=\,n$ (respectively, $ht_\SSS (G)\,=\,n$), then there
exists a complex projective manifold $M\,\in\, \HC$ (respectively,
$M\,\in\, \SSS$) of complex dimension at most $n$ realizing it.
\end{rmk}

For the category $\SSS$ we have a stronger result.

\begin{prop}\label{cx=htstein}
Assume that $-\infty \, <\,ht_\SSS (G)\,=\,n \,<\, \infty$. Then there exists a
complex projective manifold $M$ of complex dimension exactly
$n$ realizing it.  Further, no smooth 
projective manifold $M$ of complex dimension less than $n$ realizes
$ht_\SSS (G)$. Hence $hdim_\SSS(G)\,=\, 2n$.

Also, the universal cover $\til M$ of the above complex projective manifold $M$ is
homotopy equivalent to a wedge of $n$--spheres.
\end{prop}

\begin{proof}
By Lemma \ref{cx=ht} and Remark \ref{cx=htrmk}, there exists a
smooth complex projective manifold $M$ of dimension at most
$n$ realizing $ht_\SSS (G)$. The first assertion in the proposition follows from the
second assertion of Lemma \ref{cx=ht}.

To prove the second assertion, suppose that there exists a 
complex projective manifold $M$ of dimension $m$ less than
$n$ realizing $ht_\SSS (G)$. Then by the definition of $ht_\SSS (G)$,
we have $\pi_i(M) \,=\, 0$ for $2\,\leq\, i \,\leq \,m$. Hence
$\pi_i(\til{M}) \,=\, 0$ for $1\,\leq\, i \,\leq\, m$.
But $H_i (\til{M}) \,=\,0$ for $i \,>\, m$ by Theorem \ref{af}. Hence
$\pi_i(\til{M}) \,=\, 0$ for all $i$ implying that $M$ is a $K(G,1)$ space.
But this contradicts the hypothesis that
$ht_\SSS (G)\,=\,n \,<\, \infty$. Therefore, the dimension of a complex projective manifold
realizing $ht_\SSS (G)$ is at least $n$.

Combining the first two assertions we have $hdim_\SSS(G) \,=\, 2n$.

The last assertion now follows from the last statement of Theorem \ref{af}, which says
that $\til M$ is homotopy equivalent to an $n$--complex $K$. Since $\pi_i (K) \,=\, 0$ for
$i\,<\,n$, the $n$--complex $K$ must be homotopy equivalent to a wedge of $n$--spheres.
\end{proof}

\subsubsection{Lefschetz Hyperplane Theorem for Symplectic Manifolds}\label{lhptsp}

The following fundamental Theorem of Donaldson \cite{don1} generalizes the   Lefschetz Hyperplane Theorem to the context of symplectic manifolds.

\begin{theorem}[{\cite[Theorem 1; Proposition 39]{don1}}]\label{don-lhpt}
Let $(V,\omega)$ be a compact symplectic manifold of dimension $2n$, and suppose
that the de Rham cohomology class
$[\frac{\omega}{2\pi}]\,\in\, H^2(V,\, {\mathbb R})$ lies in the integral
lattice $H^2(V,\, {\mathbb Z})/Torsion$. Let $h\,\in\, H^2(V,\, {\mathbb Z})$
 be a lift of $[\frac{\omega}{2\pi}]$ to an integral class. Then for sufficiently
large integers $k$ the Poincar\'e dual of $kh$ in 
$H_{2n-2}(V,\, {\mathbb Z})$ can be
realized by a symplectic submanifold $W\,\subset\, V$.

Let $L\,\longrightarrow\, V$ be a complex line bundle over  $V$ with compatible
almost complex structure, and
with first Chern class $c_1(L) \,=\, [\frac{\omega}{2\pi}]$.
Let $W_k$ be the zero-set of a section $s$ of $L^{\otimes k}$. When $k$ is
sufficiently large, the general $W_k$ is a submanifold and the
inclusion $i\,:\, W_k \longrightarrow V$ induces an isomorphism on homotopy groups $\pi_p$
for $p < n - 2$ and a surjection on $\pi_{n-1}$.\end{theorem}

As an immediate corollary, we have the following. The proof of it
is an exact replica of the proof of Lemma \ref{cx=ht} and we omit it.

\begin{cor}\label{sp=ht}
If $ht_\SP (G)\, <\, \infty$, then there exists
a symplectic manifold $M$ realizing it such that the
 dimension of $M$ is at most $2ht_\SP (G)$; in other words,
$hdim_\SP(G)\,\leq\, 2ht_\SP (G)$.
\end{cor}

\subsection{Restrictions from formality}\label{formality}

We refer to \cite{ps} for a survey of formality of spaces and algebras relevant
to our context.  
Recall that a  commutative differential graded algebra (cdga for short) $(A\, , d_A )$ is
formal if it is weakly equivalent to its cohomology algebra, endowed with
the zero differential. Also a group $G$ is formal when the Eilenberg--MacLane space $K (G, 1)$ is formal.

\begin{definition}[\cite{macinic}] 
A cdga $(A\, , d_A )$ is $q$--stage formal if there is a zig-zag of
homomorphisms connecting $(A\, , d_A )$ to
$(H^\ast(A\, , d_A )\, , 0)$, with each one of these homomorphisms inducing an isomorphism in cohomology up to degree
$q$, and a monomorphism in degree $q + 1$. 

A group $G$ is  $q$--stage formal when the
Eilenberg--MacLane space $K (G, 1)$ is $q$--stage formal.
\end{definition}

The next theorem is a special case of a Theorem of Macinic.

\begin{theorem}[\cite{macinic}]\label{mac}
Suppose that $M$ is a  formal space but $\pi_1 (M )$ is not $q-$stage formal for some
$q\,\geq\, 2$. Then $\pi_i (M )\,\neq\, 0$ for some $i\,\in\, [2 \, ,q]$.
\end{theorem}

The following bound on $ht_\KK(G)$ is an immediate consequence.

\begin{prop} Let $G$ be a K\"ahler group. Let $q_0$ be the smallest value of $q$ such that $G$ is  not $q-$stage formal.
Then $ht_\KK(G)\,\leq\, q_0$. \label{formalKdim} \end{prop}

\begin{proof} Let $M$ be any compact K\"ahler manifold with fundamental group $G$.
A fundamental result of Deligne, Griffiths, Morgan and Sullivan
asserts that $M$ is a formal space (cf. \cite[Chapter 3]{abckt}).
Hence using Theorem \ref{mac} we have $\pi_i (M )\,\neq\, 0$ for some
$i\, \in\, [2\, , q_0]$. So $ht_\KK(G)\,\leq\, q_0$ by Definition \ref{htd}.
\end{proof}

\section{Consequences in Group Cohomology}\label{kk2}

Gromov \cite{gro-free} showed that $H^1(G,\,\ZG)\,=\,0$ (Theorem \ref{1end})
for any infinite K\"ahler group $G$.
In this section we start looking at higher group cohomologies. The guiding questions are Question \ref{tol-h2}
and Conjecture \ref{tolcon}.

\subsection{Cohomology with local coefficients}\label{specs}

We first consider the Leray--Serre cohomology spectral sequence of the 
classifying maps $M \,:=\, \til{M}/G \,\longrightarrow\, K(G,1)$
 for the principal $G$--bundle
$\til{M}\,\longrightarrow\, M$ (cf. \cite[p. 286]{hu}, \cite[Theorem 2.2]{dyer},
\cite[Proposition 1]{klingler}, \cite[Proposition 3.5]{bm-cd}).

\begin{prop} \label{leraycoh}
Let $M$ be a smooth complex projective variety with 
Stein universal cover $\til M$. Assume that $\pi_i(M) \,=\, 0$ for
all $1<i<n$. Also, assume that $G\,:=\, \pi_1(M)$ is 
torsion-free. Let $R$ be any left $\ZG$--module. Then 
\begin{enumerate}
\item $H^p(G,\,R)\,= \,H^p(M,\, R)\,$  for $0\,<\,p\,<\,n$, and
\item there is an exact sequence of $G$-modules,
$$
0 \,\longrightarrow\, H^n(G,\, R) \,\longrightarrow\, H^n(M,\, R)
\,\longrightarrow\, H^n(\til{M},\, R)^G \,\longrightarrow\,
H^{n+1}(G, \,R)
$$
$$
\longrightarrow\, H^{n+1}(M,\, R)\,
\longrightarrow \,H^1(G,\, H^n(\til{M},\, R))
\,\longrightarrow\, H^{n+2}(G,\, R) \,\longrightarrow\, H^{n+2}(M,\, R) 
$$
$$
\longrightarrow \, H^2(G,\, H^n(\til{M},\, R))\,
\longrightarrow\, H^{n+3}(G,\, R)\, \longrightarrow\,\cdots \,
\longrightarrow\,
H^{n-1}(G,\, H^n(\til{M},\, R))
$$
$$
\longrightarrow\, H^{2n}(G,\, R)\,
\longrightarrow\, H^{2n}(M,\, R)\, \longrightarrow \, H^n(G,\, H^n(\til{M},\, R))\,
\longrightarrow\, H^{2n+1}(G,\, R)\, \longrightarrow\,
0\, .
$$
\item $H^{p+n+1}(G,\, R)\,=\, H^{p}(G, \,H^n(\til{M},\, R))$ for
all $ p\,\geq \,n+1$.
\end{enumerate}
\end{prop}

\begin{proof}
Let $\til{M}\,\longrightarrow \,M$ be a principal $G$--bundle. Take
a $K(G,1)$ space $K$; its universal cover $\widetilde K$ is 
contractible. Let $f\,:\, M\,\longrightarrow\, K$ 
be a classifying map. Let $$\til{g}\,:\, \til{M} \times 
\til{K}\,\longrightarrow\, \til{K}$$
be the natural projection. The group $G$ acts on
$\til K$ through deck transformations, and it acts on
$\til{Y} \times \til{K}$ via the diagonal action.
Since $\til g$ is equivariant with respect to these
actions, it induces a map
$$
g\,:\, W\, :=\, (\til{M}\times\til{K})/G\, \longrightarrow \,
{\widetilde K}/G \,=\, K\, .
$$
The fibers of $g$ are homotopy equivalent to $\til Y$
(see \cite[pp. 285--286]{hu} for more details).

Note that $H^i(\til{M},\, R)\,= \,0\,$ for $i>n$  by  Theorem \ref{af}. Hence 
$H^i(\til{M},\, R) \,=\, 0$ for $i \,\neq\, 0, n$. 

The Leray--Serre cohomology spectral sequence for the above fibration with 
local coefficients~$R$ gives
$$H^p(K,\, (H^q(\til{M} ,\, R)) \,\Longrightarrow \,H^{p+q}(M,\, R)\, ,$$ and 
hence $$H^p(G,\, (H^q(\til{M},\, R))\,\Longrightarrow\, H^{p+q}(M,\, R)$$
since $K$ is a $K(G,1)$ space.

As $H^i(\til{M},\, R) \,=\, 0$, $i \,\neq\, 0,\, n$, it follows that $E_2^{p, 0}\,
=\,H^p(G,\, R)$ and $ E_2^{p,n}\,=\, H^p(G,\, H^n(\til{M},\, R))$ are 
the only (possibly) non-zero $E_2^{p, q}$ terms. Since $E_2^{p, i}\,=\,0$, for $0<i<n$, the 
differentials $d_2\, \cdots, d_n=\,0$. Also, the differentials $d_i$ are zero for 
$i\,>\,n+1$. Thus $d_{n+1}$ is the only (possibly) non-zero differential.
Hence $$E_{n+1}\,= \cdots =\,E_2 ~\, \text{ and }\, ~E_{n+2}\,=\,E_{n+3}\,=\,\cdots\,= \,
E_{\infty}\, ,$$ and also
\begin{itemize}
\item $E_{\infty}^{0, 0} \,=\, H^0(G,\, H^0(\til{M},\, R))\,=\,
H^0(G,\, R)\,=\,R^G$ (see \cite[p. 58]{brown} for instance), 

\item $E_{\infty}^{p,0}\,=\, H^p(G,\, H^0(\til{M},\, R))\,=\, H^p(G,\,R)$, for $p=1, \cdots, n$,

\item $E_{\infty}^{p, q}\,=\,H^p(G,\, H^q(\til{M},\, R))\,=\,H^p(G,\, 0)\,
=\,0$, for $q\, \neq\, 0,n$.
\end{itemize}

Further, we have the following exact sequences:
\begin{itemize} 
\item $0\,\longrightarrow\, E_{\infty}^{0,n}\,\longrightarrow\, H^n(\til{M},
\,R)^G \,\stackrel{d_{n+1}}{\longrightarrow}\,
H^{n+1}(G,\, R)) \,\longrightarrow\, 0$,

\item $0\,\longrightarrow\, H^{p-n-1}(G,\, H^n(\til{M},\, R))\,
\stackrel{d_{n+1}}{\longrightarrow}\, H^p(G,\, R)\,\longrightarrow\,
E_{\infty}^{p, 0}\,\longrightarrow\, 0$, for all $p \,\geq\, {n+1}$, and

\item $0\,\longrightarrow\, E_{\infty}^{p, n}\,\longrightarrow\, H^p(G,\,
H^n(\til{M},\, R))\,\stackrel{d_{n+1}}{\longrightarrow}\,
H^{p+n+1}(G,\, R)\,\longrightarrow\, 0$, for all $p\,\geq\, 1$.
\end{itemize}

The above descriptions of the $E_{\infty}^{p,q}$ terms can be assembled
to produce the following  exact sequences for the fibration:
$$
0\,\longrightarrow\,H^p(G,\,R)\,\longrightarrow\,H^p(M,\, R)\,\longrightarrow
\,0
$$
(assembling $E_{\infty}^{0,p}$ and $E_{\infty}^{p,0}$), for $0<p<n$,
$$
0\,\longrightarrow\,H^n(G,\,R)\,\longrightarrow\,H^n(M,\, R)\,\longrightarrow
\,(H^n(\til{M},\, R))^G\,\stackrel{d_{n+1}}{\longrightarrow}\, H^{n+1}(G,\, R))
$$
(assembling $E_{\infty}^{0,n}$ and $E_{\infty}^{n,0}$), and
$$
H^{p-n-1}(G,\,H^n(\til{M},\,R))\,\stackrel{d_{n+1}}{\longrightarrow}\, H^p(G,\,
R)\,\longrightarrow \,H^p(M,\,R)\,\longrightarrow\, H^{p-n}(G,\,
H^n(\til{M},\, R))
$$
$$
\stackrel{d_{n+1}}{\longrightarrow}\, H^{p+1}(G,\, R))~\,
$$
(assembling $E_{\infty}^{0,p}$ and $E_{\infty}^{p,0}$), for all  $ p
\,\geq\, {n+1}\,$.

Since $H^p(M, R)\,=\,0$ for all $p\,>\,2n$, we immediately get from
the above  exact sequence that
$$H^{p+1}(G,\, R)\,=\, H^{p-n}(G, \,H^n(\til{M},\, R))$$ for all $ p\,\geq \,2n+1$.
\end{proof}

\begin{rmk} \label{leraycohrmk} Note that in the proof of Proposition \ref{leraycoh}, the Stein property of $\til M$ is only used to conclude
that the homology of $\til M$ is concentrated in the middle dimension. Thus, Proposition \ref{leraycoh} and its conclusions remain valid when 
$M$ is only a closed orientable manifold of dimension $2n$ such that $H_i(\til{M}) =0$ for $i\neq 0, n$. \end{rmk}

\begin{cor}\label{steincd} If $ht_\SSS (G)\,=\,n$, then $cd (G)\,\geq\, n$.
\end{cor}

\begin{proof}
Choose $R = \ZG$ in Proposition \ref{leraycoh}. Then $H^p(G,\,\ZG)\,=\, H^p(M,\,\ZG)$ for
$p\,<\,n$. Now $H^p(M,\,\ZG)\,=\, H^p_c(\til{M})$ and by Poincar\'e Duality
$H^p_c(\til{M})\,=\,H_{2n-p}(\til{M})$. Further, by Theorem \ref{af}, we have
$H_{2n-p}(\til{M})\,=\, 0$ if $p\,<\,n$. Hence $H^p(G,\,\ZG)\,=\, 0$ for $p\,<\,n$. The
corollary follows.
\end{proof}

Now by Proposition \ref{cx=htstein}, if $ht_\SSS (G) \,=\,n$ is finite, then
there exists a smooth complex projective variety $M$ of complex dimension $n$ realizing $ht_\SSS (G)$  with 
Stein universal cover $\til M$.

\begin{cor} \label{leraycohstein}
Suppose $ht_\SSS (G)\,=\,n$ is finite and $M$ realizes $ht_\SSS (G)$.
Then 
\begin{enumerate}
\item $H^p(G,\,\ZG)\,= 0$ for $0\,<\,p\,<\,n$,

\item there is an exact sequence of $G$-modules,
$$
0 \,\longrightarrow\, H^n(G,\, \ZG) \,\longrightarrow\, H^n(M,\, \ZG)
\,\longrightarrow\, (H^n(\til{M},\, \ZG))^G \,\longrightarrow\,
H^{n+1}(G, \,\ZG) \,\longrightarrow\, 0\, ,
$$

\item 
$$
H^i(G,\, H^n(\til{M},\, \ZG))
\,= H^{n+i}(G,\,\ZG) \, \, 1 \leq i \leq n-2\, , $$
 
\item there is an exact sequence of $G$-modules,
$$ 0 \longrightarrow\,
H^{n-1}(G,\, H^n(\til{M},\, \ZG))
\,\longrightarrow\, H^{2n}(G,\, \ZG) \,\longrightarrow\, \Z
$$
$$
\longrightarrow \, H^n(G,\, H^n(\til{M},\, \ZG))\,
\longrightarrow\, H^{2n+1}(G,\, \ZG)\, \longrightarrow\,
0\, ,$$

\item $H^{p+n+1}(G,\, \ZG)\,=\, H^{p}(G, \,H^n(\til{M},\, \ZG))$ for all $ p\,\geq \,n+1$.
\end{enumerate}
\end{cor}

\begin{proof} This follows from Proposition \ref{leraycoh} putting $R=\ZG$ and using the following two facts:
\begin{enumerate}
\item $H^{n+i}(M,\, \ZG)=H^{n+i}_c(\til{M})  = H_{n-i}(\til{M})=\pi_{n-i}(\til{M})=0$ for $0< i<n$.
\item $H^p(M,\, \ZG) = H^{p}_c(\til{M})=H_{2n-p}(\til{M})= 0 $ $0< p <n$, where the last equality is from Theorem \ref{af}
\item $H^{2n}(M,\, \ZG)  = H^{2n}_c(\til{M})= \Z$
\end{enumerate} \end{proof}

\begin{rmk} Note that the proof of Corollary \ref{leraycohstein} does not use the complex analytic (Stein) properties of the universal cover $\til{M}$.
It uses only the following topological facts: \\
a) $M$ is a closed orientable $2n$--dimensional manifold, \\
b) $\til M$ is homotopy equivalent to a wedge of $n-$spheres. \\ Thus the conclusions of Corollary \ref{leraycohstein} hold whenever $G= \pi_1(M)$ with $M$
satisfying properties (a), (b). \label{leraycohsteinrmk} \end{rmk}

\subsection{Higher Group Cohomologies}\label{hngzg}

\begin{prop} If $ht_\SSS(G)\,>\, 2$, then either $H^2(G,\,\ZG)\,=\,0$ or $G$ is the fundamental group of a Riemann surface. \label{h2stein} \end{prop}

\begin{proof} We first deal with the case where $ht_\SSS(G)\,=\,n \,>\, 2$ is finite. 
Let $M$ be a projective variety 
of (complex) dimension $m$ realizing the homotopical height.  By Corollary \ref{surf},
we have $m \,>\, 2$.

Then by Proposition \ref{leraycoh}, $H^2(G, \,\ZG)
\,=\, H^2(M,\,\ZG)$ and the latter group is isomorphic to $H^2_c(\til{M})$. By Poincar\'e Duality
and Theorem \ref{af}, $H^2_c(\til{M})\,=\,H_{2m-2}(\til{M})\,=\,0$ and
so $H^2(G,\, \ZG)\,=\,0$.

Next suppose that $ht_\SSS(G)$ is infinite. Then one of the following two holds:
\begin{enumerate}
\item[(a)] For all $m \,>\, 2$, there exists $n\,\geq\, m$
and  a projective variety  $M\,\in\, \SSS$ with $\pi_1(M)\,=\,G$ such that
$\pi_i(M)\,= \,0$ for all $1\,<\,i\,<\,n$, but $\pi_n(M) \,\neq\, 0$. The same proof
as above goes through in this case to show that $H^2(G, \,\ZG) \,=\,0$.

\item[(b)] There exists a complex projective variety  $M \,\in \,\SSS$ of (complex)
dimension $m$ such that $M$ is a $K(G,1)$ space.
Then from Proposition \ref{ndim}, this $G$ is a PD(2m) group.
If the complex dimension of $M$ is greater than one, then again $H^2(G,\,\ZG)\,=\,0$.
\end{enumerate}

The only case left is that the complex dimension of $M$ is equal to one, when $M$ is a Riemann surface and $G$ 
is the fundamental group of a Riemann surface.
\end{proof}

The following slightly weaker statement holds in all dimensions.

\begin{prop} If $ht_\SSS(G)\, >\, n$, then either $H^n(G,\,\ZG)\,=\,0$ or $G$ is
a {\rm PD}(n) group, in which case $H^n(G,\,\ZG)\,=\,\Z$. \label{hnstein}
\end{prop}

\begin{proof}
As before,  first suppose that $ht_\SSS(G)$ is finite. Let $M$ be a projective variety 
of complex dimension $m$ realizing the homotopical height. From
Proposition \ref{ndim} we have $m \,>\, n$.

Again, as in the proof of Proposition \ref{h2stein},
we have $$H^n(G, \,\ZG) \,= \,H^n(M,\,\ZG) \,= \,H^n_c(\til{M})\,=\,
H_{2m-n}(\til{M})\,=\,0$$ and so $H^n(G, \,\ZG) \,=\,0$.

If on the other hand, $ht_\SSS(G)$ is infinite, then, one of the
following two holds:
\begin{enumerate}
\item[(a)] For all $m \,>\, n$, there exists $k\,\geq\, m$
and  a projective variety $M\,\in\, \SSS$ with $\pi_1(M)\,=\,G$ such that $\pi_i(M)\,=
\,0$ for all $1\,<\,i\,<\,k$, and $\pi_k(M)\,\neq\, 0$. The same proof
as above goes through in this case to show that $H^n(G, \,\ZG)\,=\,0$.

\item[(b)] There exists a complex projective variety  $M\,\in\, \SSS$ of (complex)
dimension $m$ such that $M$ is a $K(G,1)$ space.
Then, again from Proposition \ref{ndim}, this $G$ is a PD(2m) group.
If $n \,\neq\, 2m$, then again $H^n(G,\,\ZG)\,=\,0$.
Otherwise, $n\,=\,2m$, $G$ is a PD(n) group, and $H^n(G,\,\ZG)\,=\,\Z$.
\end{enumerate}
This completes the proof.
\end{proof}

We next investigate $H^2(G,\,\ZG)$ when $ht_\SSS(G)\,=\, 2$.

\begin{prop} Suppose that $ht_\SSS(G)\,=\,2$, and $M$ is a smooth complex projective
surface realizing the homotopical height of $G$.
Assume that the action of $G \,=\, \pi_1(M)$ on $\pi_2(M)$ is the trivial one. Then
either $H^2(G,\, \ZG)\,=\,0$ or $G$ is virtually the fundamental group of a
Riemann surface. \label{h2stein1}
\end{prop}

\begin{proof} We use Proposition \ref{leraycoh} with $n\,=\,2$. 

{}From Poincar\'e Duality and the Hurewicz' Theorem, 
$$H^2(M,\,\ZG)\,=\, H^2_c(\til{M})\,=\,H_2(\til{M})\,=\,\pi_2(M)$$
and the latter is a trivial $\ZG$ module by hypothesis.
By Theorem \ref{af}, the group $\pi_2(M)$ is free abelian.

The long exact sequence of Proposition \ref{leraycoh} shows that $H^2(G,\, \ZG)$ is a
$\ZG$--submodule of $H^2(M,\,\ZG)$. Hence $H^2(G,\, \ZG)$ is a trivial $\ZG$ module. 

Then either $H^2(G,\,\ZG)\,=\,0$ or it contains a copy of the trivial
$\ZG$-module $\Z$.
The latter case forces $G$ to be  virtually the fundamental group of a Riemann surface by
Kleiner's Theorem \ref{kleiner}.\end{proof}

\subsection{Groups with a quasi-projective $K(G,1)$ space}\label{kollar}

We shall denote the class of smooth complex quasiprojective varieties
as $\QP$. We recall Koll\'ar's question \ref{kollar-qn} below and observe in this subsection a simple but curious connection with Carlson-Toledo's Conjecture \ref{tolcon}. 

\begin{qn}[\cite{kollar-shaf}]\label{kg1} If a torsion-free group
$G\,=\,\pi_1(M)$ for some $M\,\in \,\PP$, does there exist a $K(G,1)$ space in $\QP$?
\end{qn}

The only known set of  counter-examples to Question \ref{kg1} arise from some recent examples due
to Dimca, Papadima and Suciu \cite{dps-bb} (see Theorem \ref{dps-gen-th} below for generalizations). We believe therefore that the class of groups admitting a quasiprojective 
$K(G,1)$ space deserves special attention. We denote the class of  quasiprojective 
$K(G,1)$ spaces as $\QPK$. 

\begin{prop} Suppose $M\,\in\, \QPK$ contains a positive dimensional projective
subvariety. Then $b_2(\pi_1(M)) \,>\, 0$. \label{qprojb2}
\end{prop}

\begin{proof}
Let $N\,\subset\, M$ denote a positive dimensional projective subvariety of $M$. Let
$$M\, \subset\, \overline{M}\,\subset\, \C{\mathbb P}^N$$ be the closure of an embedding
of $M$. Restrict the usual Fubini--Study K\"ahler class $\omega$ on $\C{\mathbb P}^N$ to
$\overline{M}$ and thence to $N$. Let $Z\,\subset\, N$ be an algebraic curve in $N$ (obtained by
taking enough hyperplane sections of $N$, say). Then
\begin{equation}\label{exz}
\int_Z \omega \,\neq\, 0\, .
\end{equation}

Let $i_0\,:\, M \,\hookrightarrow\, \overline{M}$,
$i\,:\, N \,\hookrightarrow\, M$ and $j\,:\, Z \,\hookrightarrow\, N$ be the
inclusions. From \eqref{exz} it
follows that
$$
(i_0 \circ i\circ j)^\ast\,:\, H^2(\overline{M},\, \C)\,\longrightarrow\, H^2(Z,\, \C)
$$
is non-zero. Therefore, $(i\circ j)^\ast\,:\, H^2(M,\, \C)\,\longrightarrow
\,H^2(Z,\, \C)$ is non-zero. Hence $b_2(M)$ is nonzero.
Since $M$ is a $K(\pi_1(M),1)$ space, this implies that $b_2(\pi_1(M)) > 0$.
\end{proof}

\begin{rmk} All known examples of projective groups $G$ having quasiprojective $K(G,1)$'s are constructed in a way that the quasi-projective
$K(G,1)$ contains a projective subvariety with $G$ as its fundamental group (see \cite[Ch. 8]{abckt} for instance). They thus satisfy the hypothesis
of Proposition \ref{qprojb2} \end{rmk}

\begin{cor} Suppose $M\,\in\, \QPK$ admits a nonconstant holomorphic map $\phi$ from a
compact complex analytic manifold $N$. 
Then $b_2(\pi_1(M)) \,>\, 0$. \label{qprojb2cor} \end{cor}

\begin{proof} Consider $M\,\subset\, \overline{M}\subset \C{\mathbb P}^N$ as in the proof of
Proposition \ref{qprojb2}. Then, by Chow's Theorem, the image $\phi(N)\,\subset\, M\,
\subset\, \C{\mathbb P}^N$ is algebraic. Since $\phi$ is
non-constant, $\phi(N)$ is positive dimensional, and hence Proposition \ref{qprojb2}
completes the proof.
\end{proof}

\begin{rmk}\label{jou-trick}
If $G$ is an infinite K\"ahler group admitting a quasiprojective $K(G,1)$
space $M$, in order to show that $b_2(G) \,>\, 0$, according to Corollary \ref{qprojb2cor}
it suffices to show that $M$ admits a nonconstant holomorphic map from some compact
complex analytic manifold.

Recall Jouanolou's trick \cite{jo} (cf. Lemma \ref{aff=qp}), where he showed that for any quasi-projective variety $M$,
there exists an affine variety $W$  mapping surjectively onto $M$ with contractible
fibers. Thus, if a group $G$ admits a quasi-projective $K(G,1)$ space, it also
admits an affine  $K(G,1)$ space. Since affine varieties do not
admit any nonconstant holomorphic maps from any compact complex analytic manifold,
the restriction on $M$ in Corollary \ref{qprojb2cor} is nontrivial.
\end{rmk}

Instead of demanding that $M$ admits  a nonconstant holomorphic map from a compact complex analytic manifold, we could
also demand that $M$ receives a symplectic map $\phi$ from a symplectic manifold of
positive dimension.

\begin{prop} Suppose $M\,\in\, \QPK$ equipped with some symplectic structure
 receives a symplectic map $\phi$ from a compact symplectic map manifold $N$
of positive dimension. Then $b_2(\pi_1(M))\,>\, 0$. \label{qprojb2symp}
\end{prop}

\begin{proof}
We have a compact symplectic manifold $(N\, ,\omega)$ of positive dimension,
a symplectic form $\eta$ on $M$, and a map $\phi\, :\, N\, \longrightarrow\, M$
such that $\phi^\ast\eta\,=\, \omega$.
Since the cohomology class $0\,\not=\, [\omega]\,\in\, H^2(N,\, {\mathbb R})$, it
follows that $[\eta]\,\in\, H^2(M,\, {\mathbb R})$ is nonzero, and hence
$b_2(\pi_1(M))\,= \,b_2(M)$ is nonzero.
\end{proof}

\begin{rmk}
Examples of fundamental groups $G$ of smooth projective varieties admitting a
quasiprojective $K(G,1)$ space are constructed by Toledo in \cite{tol1} and in
Theorem 8.8 of \cite{abckt} using work of Goresky--MacPherson
\cite[Section 2]{gormac}. These examples satisfy the hypotheses in both Proposition 
\ref{qprojb2} and Proposition \ref{qprojb2symp}.
\end{rmk}

\section{Finiteness Properties}\label{fin}

In this section we compute explicitly the $\SSS$--homotopical
dimension of certain projective groups discovered by Dimca, Papadima and Suciu \cite{dps-bb} (see also \cite{suciu-imrn}). We then show how to generalize and topologically characterize them.
Our techniques replace the constructions using characteristic varieties  in \cite{dps-bb}
by the group cohomology computations developed in Section \ref{kk2}.

\subsection{Projective groups violating $FP_r$}
The examples in \cite{dps-bb} are inspired by certain
examples due to Bestvina and Brady \cite{BB}, but adapted to the context of K\"ahler groups.

\begin{theorem}[{\cite[Theorem A]{dps-bb}}]\label{dpsgps}
For every $r\,\ge \,3$, there is an $(r-1)$--dimensional smooth 
complex projective variety $H_r$ with fundamental 
group $G_r$ satisfying the following conditions:
\begin{enumerate}
\item\label{ko1}
The  homotopy group $\pi_i(H_r)$ vanish for all $2\,\leq\, i\,\leq\, r-2$, 
while $\pi_{r-1}(H)\,\neq\, 0$.
\item \label{ko0}
The universal cover $\widetilde{H_r}$ of $H_r$ is a Stein manifold.
\item \label{ko2}
The group $G_r$ is of type $FP_{r-1}$, but not of type $FP_r$.
\item \label{ko3}
The group $G_r$ is not commensurable (up to finite kernels) 
to any group having a classifying space of finite type.
\item \label{ko4} The homotopy group  $\pi_{r-1}(H_r)$ is an infinitely generated
free $\ZG_r$--module.
\end{enumerate}
\end{theorem}

Proposition \ref{cx=htstein} shows that the above
Conditions \ref{ko1}, \ref{ko0} are satisfied by a large class of manifolds:
Take any $G$ that appears as the fundamental group of a  manifold in $\SSS$. Then
any manifold $M\,\in\, \SSS$ realizing $ht_\SSS(G)$ satisfies 
Conditions \ref{ko1}, \ref{ko0}. Further, any hyperplane section of such an $M$ also  satisfies 
Conditions \ref{ko1}, \ref{ko0}.
We now observe that the groups in Theorem \ref{dpsgps} have precisely identifiable $\CC$-homotopical height.

\begin{prop} With the notation of Theorem \ref{dpsgps}, we have
$$ht_\KK(G_r)\,=\, ht_\PP (G_r)\,=\, ht_\HC(G_r)\,=\, ht_\SSS (G_r)\,=\, r-1\, .$$
Also, $b_2(G_r) \,>\, 0$. \label{exactdim}
\end{prop}

\begin{proof} 
By Condition \ref{ko1} of Theorem \ref{dpsgps}, we have $ht_\KK(G_r)\,\geq\, r-1$. If
$ht_\KK(G_r)\,\geq\, r$, then there exists a smooth complex projective variety
$M$ such that $\pi_i(M)$ vanish for $2\,\leq\, i\,\leq\, r-1$. By attaching cells of dimension greater than $r$ to $M$, we can
construct a $K(G,1)$ space with finite $r$--skeleton (equal to the $r$--skeleton of $M$). This contradicts 
Condition \ref{ko2} of Theorem \ref{dpsgps} which says that $G$ is not of type $FP_r$.

Since $H_r$ is Stein and we have $$ht_\KK(G_r)\,\geq\, ht_\PP (G_r)
\,\geq\, ht_\HC(G_r)\,\geq\, ht_\SSS (G_r)$$
(a general fact) the first claim follows. 

For $r\,>\,3$, we have $b_2(G_r) = b_2(H_r) > 0$ as $H_r$ is K\"ahler. It remains to show that $b_2(G_3)\,>\, 0$.

The authors of \cite{dps-bb} construct a surjective holomorphic map of the product $W_r$ of $r$ Riemann surfaces
to an abelian variety $V_r$ of dimension $r$. The abelian variety
$V_r$ is then mapped to an elliptic curve $E$ and the
generic smooth fiber of the composition from $W_r$ to $E$ is $H_r$ with fundamental group $G_r$. It follows that
there is a surjective map from $G_3$  to $\Z^{6-2}\,= \,\Z^4$. Therefore, $b_1(G_3)
\,>\, 0$. Hence the Albanese variety of $H_3$
is non-trivial and so (see \cite{kkm} for instance)  $b_2(G_3)\,>\, 0$.
\end{proof}

\begin{rmk} Proposition \ref{exactdim} shows  that the examples in \cite{dps-bb} (Theorem \ref{dpsgps})  {\bf do not}
 provide counterexamples to Question \ref{kollar-qn1} (3). 
The authors of \cite{dps-bb} observe that $cd (G_3)$ is strictly greater than $ht_\KK(G_3)$ $(\,=\,2)$ as $G_3$ is not of type $FP_3$. \end {rmk}

We denote the free $\ZG_r$ module $\pi_{r-1}(H_r)$ by $\ZG_r^\alpha$ for $\alpha$ a countably infinite indexing set.
Then $H^{r-1}(\til{M},\,  \ZG_r) = Hom_\Z(\pi_{r-1}(H_r)\,  ,\ZG_r) $ and we denote this 
 $\ZG$ module  by $\ZG_r^{\alpha \ast}$.  Also, $(Hom_\Z(\pi_{r-1}(H_r)\,  ,
\ZG_r))^{G_r} = Hom_{\ZG_r}(\pi_{r-1}(H_r)\,  ,\ZG_r)$,
which we denote as $\ZG_r^{\alpha 0}$.
Combining Condition \ref{ko4} of Theorem \ref{dpsgps} with Proposition \ref{exactdim} and
Corollary \ref{leraycohstein}, we obtain the following information regarding the cohomology groups
$H^n(G,\, \ZG_r)$.

\begin{prop} \label{hrleraycoh} With the notation of Theorem \ref{dpsgps},
the following hold:
\begin{enumerate}
\item $H^p(G,\,\ZG_r)\,=\, 0$ for $0\,<\,p\,<\,r-1$.

\item There is an exact sequence of $G$-modules,
$$
0 \,\longrightarrow\, H^{r-1}(G,\, \ZG_r) \,\longrightarrow\, \ZG_r^\alpha
\,\longrightarrow\, \ZG_r^{\alpha 0} \,\longrightarrow\,
H^{r}(G, \,\ZG_r) \,\longrightarrow\, 0\, .
$$

\item $H^i(G,\, \ZG_r^{\alpha \ast})\,=\, H^{r-1+i}(G,\,\ZG_r)$ if
$1\,\leq\, i \,\leq\, r-3$.  

\item There is an exact sequence of $G$-modules,
$$ 0 \longrightarrow\,
H^{r-2}(G_r,\, \ZG_r^{\alpha \ast})
\,\longrightarrow\, H^{2(r-1)}(G,\, \ZG_r) \,\longrightarrow\, \Z
$$
$$
\longrightarrow \, H^{r-1}(G,\, \ZG_r^{\alpha \ast})\,
\longrightarrow\, H^{2r-1}(G,\, \ZG_r)\, \longrightarrow\,
0\, .$$

\item $H^{p+r}(G,\, \ZG_r)\,=\, H^{p}(G, \,\ZG_r^{\alpha \ast})$ for all $ p\,\geq \,r$.
\end{enumerate}
\end{prop}

\subsection{Topological Lefschetz Fibrations and $FP_r$}

The construction in \cite{dps-bb} can be generalized in different directions. 
We shall use the purely topological structure underlying the
(complex) Morse theory argument (\`a la Lefschetz \cite{lamotke, dps-bb}) to prove that there exists
$M$ with $G=\pi_1(M)$ satisfying the hypotheses of Corollary \ref{leraycohstein} or more generally
those of Remark \ref{leraycohsteinrmk} (Conditions (a), (b)). 

Then we shall use Theorem \ref{bieri} below from group cohomology  to conclude
that   $G$ violates property $FP$. Since quasiprojective varieties have the homotopy type of finite CW complexes, it follows immediately
that such a $G$ cannot have a quasiprojective $K(G,1)$ space. Thus we have a large source of counterexamples to Koll\'ar's Question \ref{kollar-qn}.
The aim of this subsection is thus to replace the arguments in  \cite{dps-bb} involving resonance and characteristic varieties
by group cohomology computations and thus generalize those examples to a purely topological context.

We  refer the reader to \cite{don1, don2} for details of the theory of topological and symplectic Lefschetz fibrations
and adapt the definition from \cite{don2, gompf-hid} below.

\begin{defn} A {\it topological Lefschetz fibration} on a smooth, closed, oriented $2n-$manifold
$M$ consists of the following data:
\begin{enumerate}
\item   a closed orientable 2-manifold $S$,
\item a finite set of points $K= \{ b_i \} \subset M$ called the critical set,
\item a smooth map $f : M  \longrightarrow
S$ whose differential $df$  is surjective outside $K$,
\item for each critical point $x$ of $f$, there are orientation preserving coordinate
 charts about $x$ and $f (x)$ (into $\C^n$ and $\C$, respectively) in which $f$ is
 given by $f (z_1 , \cdots , z_n ) = \sum_{i=1\cdots n} z_i^2$, and
\item $f$ is injective on the critical set $K \subset X$.
\end{enumerate}

A topological Lefschetz fibration with base $S$ a surface of
positive genus will be called an {\it irrational} topological Lefschetz fibration.
A Lefschetz  fibration without singular fibers is called a {\it Kodaira fibration}.

If $S$ is an open orientable manifold instead of being closed, and $K$ is 
only discrete rather than finite, the fibration will be called an {\it open 
topological Lefschetz fibration}.\end{defn}

Note that the definition of a topological Lefschetz fibration ensures that $|f|^2$ is  a Morse function near singular fibers
and a regular fibration away from the critical values.
It turns out that the structure of  a topological Lefschetz fibration
is all that one needs for the usual Lefschetz hyperplane theorem argument (cf. \cite{lamotke}) to go through and for the proof of
 Corollary 5.4 of \cite{dps-bb} to work out in a topological context (thus replacing a hypothesis involving characteristic varieties in  \cite{dps-bb}).

\begin{prop}
Let $f: M \longrightarrow S$ be a (possibly irrational) topological Lefschetz
fibration  with $\dim M \,=\, 2n+2$, $n \geq 2$. 
Let $K$ be the finite critical set of $f$. Suppose that $\til{M}$
is $(n+1)-$connected.  Let $F$ denote the regular fiber and $N=\pi_1(F)$.  
   Then \\
a) $\pi_k(F) =0$ for $1<k<n$, \\
b) $\pi_{n}(F)$ is a free $\Z N$-module, 
with generators in one-to-one correspondence with $K\times \pi_1(S)$, and \\
c) if, furthermore, $\til M$ is contractible, then $\til F$ is homotopy
equivalent to a wedge of $n$--spheres.
 \label{dps-gen} \end{prop}

\begin{proof} The proof below is based on ideas in \cite{lamotke, dps-bb}, particularly the proofs of Lefschetz-type theorems using Morse theory.

Let $M_N$ denote the cover of $M$ corresponding to the (image of the) subgroup
$N$ $(\,=\,\pi_1(F))$ of $\pi_1(M)$
(we shall observe below that $N$ injects into $\pi_1(M)$). Then $f$ induces   a (possibly open) topological  Lefschetz fibration
 $f_N: M_N \longrightarrow \til{S} $ with regular fiber $F$, where 
 $\widetilde{S}$ is the universal cover of $S$.

It follows (cf. \cite{lamotke, don1, dps-bb}) that $M_N$ can be obtained from $F$ (up to homotopy) by attaching exactly one $n+1$ cell $\D^{n+1}_i$
for every critical value $v_i$ of $f_N$ along a 
 vanishing cycle of the Lefschetz fibration $f_N: M_N \longrightarrow \til{S} $.
This is because $f$ is assumed to be injective on the critical set $K$. Since
$n \,\geq\, 2$, we have $\pi_1(M_N) = N$.
Also, $\pi_k(M_N,\, F)\,=\,0$ for $k\,\leq\, n$.

Since $\til{M}$ is $(n+1)-$connected, we have $\pi_k(M_N)\,=\, 0$, for $1\,<\,k
\,\leq\, n+1$. The  homotopy exact sequence of the pair 
$(M_N,\, F)$ gives the following: $$  \cdots \longrightarrow \pi_k(M_N)
\longrightarrow \pi_k(M_N,F) \longrightarrow  \pi_{k-1}(F)
\longrightarrow \pi_{k-1}(M_N) \longrightarrow \cdots$$
and hence $\pi_k(M_N,F) =  \pi_{k-1}(F)$ for $2< k \leq n+1$. In particular, $\pi_{n+1}(M_N,F) =  \pi_{n}(F)$. 
Since  $\pi_k(M_N,F)\,=\,0$ for $k\leq n$, we have $ \pi_{k-1}(F)\,=\,0$ for
$2<k\leq n$, which is statement (a) of the proposition.

Further,
$\pi_{n+1}(M_N,F)$ is generated as a free $\pi_1(M_N)-$module
precisely by the $(n+1)$ cells $\D^{n+1}_i$ attached to vanishing cycles of the Lefschetz fibration.
If $$\phi\,:\,M_N\,\longrightarrow\, M$$ denotes the  covering map, then the critical set of $f_N$ is $\phi^{-1}(K)$, which in turn can be naturally identified with
$K \times \pi_1(S)$. Statement (b) of the Proposition follows.

To prove part (c), first observe that the inclusion of the regular fiber $F$ in
$M_N$ lifts to an inclusion of $\til F$ into $\til M$, because $\pi_1(M_N)\,=\, N$.
Hence $\til M$ can be obtained from $\til F$ up to homotopy by attaching the
lifts of the $(n+1)$ cells $\D^{n+1}_i$ to $\til M$. So
$\pi_n (\til{F})$ is a free abelian group with generators in one-to-one
correspondence with the indexing set $I\,=\,K\times \pi_1(M)$.
Let $V_I$ denote the wedge of a collection of $n-$spheres indexed by $I$. 
Let $H: V_I \longrightarrow \til(\til{F})$ be a map inducing an isomorphism of $\pi_n$'s. Since the contractible space
$\til M$ can be obtained from $\til F$ up to homotopy by attaching the lifts of the $n+1$ cells $\D^{n+1}_i$, it follows from the Mayer Vietoris homology
exact sequence that $H$ induces an isomorphism of  homology groups $H_i (\til{F}) = H_i (V_I)$. By simple connectivity of $\til F$ and $V_I$, and from the Hurewicz' Theorem it follows that $\til F$ and $V_I$ are homotopy equivalent.
\end{proof}

We recall now a classical theorem of Bieri.

\begin{theorem}[\cite{bieri}]\label{bieri}
Let  $1 \longrightarrow N \longrightarrow G \longrightarrow Q
\longrightarrow 1$ be a short exact  sequence of groups,  with $N$ of  type  (FP) 
and $Q$ of finite  cohomology  dimension.  Then  $G$ is a duality (respectively, PD) group  if and only if $N$ and $Q$ are duality (respectively, PD) groups. 
Further $cd(G) = cd(N) + cd(Q)$. \end{theorem}

As an immediate corollary of it we have:

\begin{cor}\label{biericor}
Let  $1 \longrightarrow N \longrightarrow G \longrightarrow Q \longrightarrow 1$ be a short
exact  sequence of groups, with both $G, Q$ 
PD groups. Further suppose that $N$ is not a PD group of finite
cohomology  dimension $cd(G)-cd(Q)$.
 Then $N$ is not of type FP. Hence, $N$ cannot have a  K(G,1) space
homotopy equivalent to a finite CW complex. In particular, $N$ cannot have a
quasiprojective K(G,1) space. \end{cor}

We combine Proposition \ref{dps-gen} with Corollary \ref{biericor} to obtain the following Theorem (and its Corollary \ref{dps-gen-cor})
providing a  source of counterexamples
to Koll\'ar's Question \ref{kollar-qn} generalizing those in \cite{dps-bb}.

\begin{theorem}
Let $f: M \longrightarrow S$ be an irrational topological Lefschetz fibration
that is not a Kodaira fibration, with $\dim M \,=\, 2n+2$, $n\,\geq\, 2$. 
Let $K$ be the finite critical set of $f$. Further suppose that $\til{M}$
is contractible.  Let $F$ denote the regular fiber and $N=\pi_1(F)$.  
   Then \\
a) $\pi_k(F) =0$ for $1<k<n$, \\
b) $\pi_{n}(F)$ is a free $\Z N$-module, 
with generators in one-to-one correspondence with $K\times \pi_1(S)$, \\
c) $\til{F}$ is homotopy equivalent to a wedge of $n$--spheres, and\\
d) $N$ cannot be of type FP; in particular, there does not exist a
quasiprojective $K(N,1)$ space. \label{dps-gen-th} \end{theorem}

\begin{proof} Statements (a), (b) and (c) follow directly from Proposition \ref{dps-gen}.\\
{\it Proof of (d):} It is at this stage that we replace the arguments in  \cite{dps-bb} involving resonance and characteristic varieties
by group cohomology computations.

Note that $\pi_1(M)$ and $\pi_1(S)$ are PD groups of dimension $(2n+2)$ and $2$ respectively.
To show that $N$ cannot be of type FP, it suffices (by Corollary \ref{biericor}) to show that $N$ cannot be a PD(2n) group. 
By Proposition \ref{leraycoh} (and Remark \ref{leraycohrmk}), there is 
an exact sequence of $G$-modules,
$$
0 \,\longrightarrow\, H^n(N,\, \Z N) \,\longrightarrow\, H^n(F,\, \Z N)
\,\longrightarrow\, H^n(\til{F},\, \Z N)^N \,\longrightarrow\,
H^{n+1}(N, \,\Z N)\,\longrightarrow\, \cdots
$$
If $N$ is a PD(2n) group, then we have $H^n(N,\, \Z N)\,=\, H^{n+1}(N, \,\Z N)\,=\,0$
because $n \geq 2$. Hence $H^n(F,\, \Z N) \,=\, H^n(\til{F},\, \Z N)^N$. 

Now,  by Poincar\'e Duality and the Hurewicz' Theorem, we have,
$$H^n(F,\, \Z N)\,=\, H^n_c(\til{F})\,=\, H_n(\til{F})\,=\,
\pi_n(\til{F}) \,=\,\bigoplus_I \Z\, ,$$
where $\bigoplus_I \Z$ denotes direct sum of a collection of copies of $\Z$ indexed by $I$.
Also, $$H^n(\til{F},\, \Z N)^N \,=\, (Hom_\Z (H_n(\til{F}), \,\Z N))^N
\,=\, (Hom_\Z (\pi_n(\til{F}),\, \Z N))^N\,=\, $$
$$\hfill Hom_{\Z N}
(\pi_n(\til{F}), \,\Z N)\,=\, Hom_{\Z N} \left(\bigoplus_I \Z,\, \Z N\right)
\,=\, \prod_I \Z N\, ,$$
where $\prod_I \Z N$ denotes direct product of  a collection of copies of $\Z N$ indexed by $I$.
Hence $$\bigoplus_I \Z \,=\, \prod_I \Z N\, .$$

Note that as an abelian group $\Z N$ is a direct sum of (a non-zero number of) copies of $\Z$.
Since $f: M \longrightarrow S$ is an irrational topological Lefschetz fibration (this is
exactly the place where we use irrationality of $f$) the set $I$ is countably
infinite.  Therefore, $\bigoplus_I \Z$ is countable and $ \prod_I \Z N$ is uncountable and the two cannot be equal. This contradiction shows that $N$ 
cannot be a PD(2n) group and hence cannot be of type FP. 
In particular, there does not exist a quasiprojective $K(N,1)$ space.
\end{proof}

\begin{cor}\label{dps-gen-cor} Let $M$ be a smooth projective variety of complex dimension $n+1$ with $n \geq  2$,  such that $\til{M}$ is Stein
and $\pi_{n+1}(M)=0$. Suppose also that $M$ realizes $ht_\SSS (\pi_1(M))$.
Let $f: M \longrightarrow S$ be an irrational  Lefschetz fibration that is not a
Kodaira fibration, and let $F$ denote the regular fiber
of the fibration. Denote $N=\pi_1(F)$.  
   Then \\
a) $\pi_k(F) \,=\,0$ for $1<k<n$, \\
b) $\pi_{n}(F)$ is a free $\Z N$-module, 
with generators in one-to-one correspondence with $K\times \pi_1(S)$, \\
c) $\til{F}$ is homotopy equivalent to a wedge of $n$--spheres,  \\
d) $N$ cannot be of type FP; in particular, there does not exist a
quasiprojective $K(N,1)$ space, and\\
e) $\til{F}$ is Stein.
\end{cor}

\begin{proof} Since $M$ realizes $ht_\SSS (\pi_1(M))$, we have $\pi_k(M)\,=\,0$ for $1<k<n+1$. By hypothesis $\pi_{n+1}(M)=0$.
Hence by Theorem \ref{af} and the Hurewicz' Theorem, $\til M$ is contractible.
Theorem \ref{dps-gen} now applies directly to prove (a)--(d).

Since $\pi_1(F)$ injects into $\pi_1(M)$, it follows that $\til F$ is a complex analytic submanifold of the Stein manifold $\til M$. Hence $\til F$ is Stein, proving 
(e). \end{proof}

\subsubsection{Products of Riemann surfaces}\label{prs}  We modify the notation of Corollary  \ref{dps-gen-cor} slightly: let $M$ be
a smooth projective variety with contractible universal cover. We now turn to the specific examples of $M$ constructed  in \cite{dps-bb}, viz.
$M$ is a product $S_1 \times \cdots \times S_r$ of $r>2$
Riemann surfaces, $S$ is a torus, and $f$ is a branched cover restricted to every factor
$S_i$. Dimca, Papadima and Suciu construct explicit such examples in \cite{dps-bb}. We provide below a criterion at the level of fundamental
groups that furnish (in a sense) all examples arising this way. The main new
ingredient  is the following theorem of Edmonds \cite{edmonds}:

\begin{theorem}[\cite{edmonds}]\label{edmonds}
A map $f: M \longrightarrow N$ of closed, orientable surfaces is homotopic
to a branched covering if and only if $$f_\ast: \pi_1(M)\longrightarrow \pi_1(N)$$ is
injective  or $${\rm deg} (f)\,>\, [\pi_1(N): f_\ast(\pi_1(M))]\, .$$
If $f_\ast: \pi_1(M)\longrightarrow \pi_1(N)$ is injective, then $f$ is homotopic
to a covering map. If 
${\rm deg}(f)\,>\, [\pi_1(N): f_\ast(\pi_1(M))]$, then $f$ is homotopic to a
branched cover map but not to a covering map. \end{theorem}

We collect together a few elementary facts.

\begin{lemma} Let $M$ be a product $S_1 \times \cdots \times S_r$ of closed,
orientable surfaces $S_i$ of positive genus with $r\geq 2$, and let $S$ be another
closed, orientable surface. Fix $(x_1, \cdots , x_r )\,\in\, S_1 \times \cdots \times
S_r$
and let $j_k\,:\, S_k \,\longrightarrow\, S_1 \times \cdots \times S_r$ be
given by $j_k(y_k) = (x_1, \cdots, x_{k-1}, y_k, x_{k+1},  \cdots, x_{r})$.
Let $f: M\longrightarrow S$ be a map such that \\
a) $f_\ast (\pi_1(M))$ $(\,\subset\, \pi_1(S))$ is not cyclic, and\\
b) $f_\ast \circ j_{k\ast} $ is not the trivial homomorphism
for any $k$.\\ Then $S$ must be a torus. \label{torus} \end{lemma}

\begin{proof}  By hypothesis (b), for all $k=1, \cdots, r$, there exists $a_k \in \pi_1(S_k)$ such that
$f_\ast \circ j_{k\ast}(a_k) $ is non-trivial. Hence
the images $f_\ast \circ j_{l\ast}( \pi_1(S_l))$ commute with $f_\ast \circ j_{k\ast}(a_k) $ for all $l \neq k$.

By hypothesis (a) $f_\ast (\pi_1(M))\, \neq\, \langle f_\ast \circ j_{k\ast}(a_k) \rangle$, the (non-trivial) cyclic group generated by
$f_\ast \circ j_{k\ast}(a_k) $. Since $\pi_1(S)$ is torsion-free, this implies that $\pi_1(S)$ contains a free abelian group of rank greater than one.
Hence $S$  must be a torus. \end{proof}

\begin{lemma} Let $M$ be a product $S_1 \times \cdots \times S_r$ of
Riemann surfaces of positive genus equipped with the product complex structure. 
Let $S$ be another Riemann surface of positive genus. Fix
$(x_1, \cdots , x_r )\,\in\, S_1 \times \cdots \times S_r$
and let $j_k: S_k \longrightarrow S_1 \times \cdots \times S_r$ be
as in Lemma \ref{torus}. Let $$f\,:\, M\,\longrightarrow\, S$$ be a map such that $f\circ j_k $ is non-constant for all $k$. Then $S$ has genus one and $f\circ j_k $ is a branched cover
 for all $k$. \label{rs}\end{lemma}

\begin{proof} Since any non-constant complex analytic map between Riemann surfaces is a branched cover, $f\circ j_k $ is a branched cover
 for all $k$.

Further, any branched cover has positive degree, and hence $[\pi_1(S):f_\ast \circ j_{l\ast}( \pi_1(S_k))]$ is finite for all $k$.
In particular, the hypotheses (a) and (b) of
Lemma \ref{torus} are satisfied. Therefore, the genus of $S$ is one. \end{proof}

We shall say that a complex analytic map $f: M\longrightarrow S$ between a product $M=S_1 \times \cdots \times S_r$ of Riemann surfaces and a Riemann surface $S$
is {\it completely non-degenerate} if $f\circ j_k $ is a branched cover for all
$k$. The next proposition characterizes homotopy classes of 
completely non-degenerate maps and shows that the examples constructed in \cite{dps-bb} account for all irrational Lefschetz fibrations
of a product of Riemann surfaces.

\begin{prop}\label{charzn}
Let $M$ be a product $S_1 \times \cdots \times S_r$ of closed, orientable
surfaces $S_i$ of positive genus with $r\geq 2$, and let $S$ be another
closed, orientable surface of positive genus. Fix $(x_1, \cdots , x_r )
\,\in\, S_1 \times \cdots \times S_r$
and let $j_k: S_k \longrightarrow  S_1 \times \cdots \times S_r$
be as in Lemma \ref{torus}.
If a (continuous) map $f: M \longrightarrow S$ is homotopic to a completely
non-degenerate complex analytic map (for some complex structures on $S_k$ and $S$)
then: \\
a) $S$ has genus one, and\\
b) for all $k$, the homomorphism $(f\circ j_k)_\ast\,:\, \pi_1(S_k)\,
\longrightarrow \,\pi_1(S)$ is either injective or
$$\deg ((f\circ j_k))> [\pi_1(N): f_\ast(\pi_1(M))]\, .$$

Conversely, if a continuous map $f: M \longrightarrow S$ satisfies (b), then $S$ has genus one and there exist complex structures on $S_k, S$
such that $f: M \longrightarrow S$ is homotopic to a completely non-degenerate complex analytic map $\phi$ for the product complex structure on $M$.
Further,  there exist complex analytic maps $\phi_k : S_k \longrightarrow S$ such that $\phi (z_1, \cdots , z_k) = m(\phi_1(z_1), \cdots , \phi_r(z_r))$,
where $m(.,.)$ denotes  addition on $S$.
\end{prop}

\begin{proof} Suppose $f$ is homotopic to a completely non-degenerate complex
analytic map. Then by Lemma \ref{rs}, the surface
$S$ has genus one. Also by hypothesis, for all $k$, the map $f\circ j_k$ is homotopic to a
branched cover. Hence by Theorem \ref{edmonds}, the homomorphism
$(f\circ j_k)_\ast\,:\, \pi_1(S_k)\,\longrightarrow\, \pi_1(S)$ is either
injective  or $\deg ((f\circ j_k))\,>\, [\pi_1(N): f_\ast(\pi_1(M))]$.

Conversely, suppose that for all $k$, the homomorphism $(f\circ j_k)_\ast\,:\,
 \pi_1(S_k)\,\longrightarrow \,\pi_1(S)$ is either injective  or $\deg
((f\circ j_k))\,>\, [\pi_1(N): f_\ast(\pi_1(M))]$.
Then again by Theorem \ref{edmonds}, for all $k$, the map $f\circ j_k$ is homotopic to a
branched cover map $\phi_k$. Hence by Lemma \ref{torus}, the surface $S$ has
genus one. Equip $S$ with some complex structure and pull back the complex
structure to $S_k$ by $\phi_k$. Equip $M$ with the resulting product complex structure $J$. Let $m(.,.)$ denote  addition on $S$.
Define $\phi: (M,J) \longrightarrow S$ by
$\phi (z_1, \cdots , z_r) = m(\phi_1(z_1), \cdots , \phi_r(z_r))$. It is easy to check that $\phi$ and $f$ induce the same map at the level of
fundamental groups (as $(f\circ j_k)_\ast\,=\,(\phi\circ j_k)_\ast$ for all $k$). Since $M$ and $S$ are Eilenberg-Maclane spaces, $\phi$ and $f$ are homotopic.
\end{proof}

We finally come to a topological characterization of irrational Lefschetz fibrations
of a product of Riemann surfaces. We continue using $j_k$ as in Proposition \ref{charzn}.

\begin{theorem}\label{charzn-th}
Let $M$ be a product $S_1 \times \cdots \times S_r$ of closed, orientable surfaces $S_i$ of positive genus with $r\geq 2$.
If there exists a product complex structure on $M$ (coming from complex structures on $S_k$) and an irrational Lefschetz fibration
$f\,:\, M \,\longrightarrow\, S$ which is not a Kodaira fibration, then \\
a) $S$ has genus one, and\\
b) ${\rm deg} ((f\circ j_k))\,> \,[\pi_1(N): f_\ast(\pi_1(M))]$ for all $k$.\\

Conversely, suppose that $S$ is a closed orientable surface of positive genus and
there exists a continuous map $f: M \longrightarrow S$ such that
$\deg ((f\circ j_k))\,> \,[\pi_1(N): f_\ast(\pi_1(M))]$ for all $k$.
Then $S$ has genus one and there exist complex structures on $S_k, S$
such that $f: M \longrightarrow S$ is homotopic to a Lefschetz fibration (for $M$ equipped with the product complex structure). 
Further,  there exist  complex analytic maps $\phi_k : S_k \longrightarrow S$ such that $\phi (z_1, \cdots , z_k) = m(\phi_1(z_1), \cdots , \phi_r(z_r))$,
where $m(.,.)$ denotes  addition on $S$. \end{theorem}

\begin{proof} For the forward direction, observe by Proposition \ref{charzn}, that
 it suffices to rule out the possibility that $(f\circ j_k)_\ast: \pi_1(S_k)
\longrightarrow \pi_1(S)$ is injective for some $k$. If not, then $f\circ j_k$ is
homotopic to a unramified covering map.
It follows that $f$ is a local submersion at all points, and hence $f$ is homotopic to a Kodaira fibration.
But a  Lefschetz fibration
 $f: M \longrightarrow S$ which is not a Kodaira fibration cannot be homotopic to a Kodaira fibration. It follows that 
$(f\circ j_k)_\ast: \pi_1(S_k)\longrightarrow \pi_1(S)$ cannot be injective
for any $k$.

We now prove the converse direction. By Proposition \ref{charzn} the surface $S$ has genus one. Also,  
for all $k$, the map $f\circ j_k$ is homotopic to a
branched cover map but not a covering map. We can further homotope $f\circ j_k$ so that each branch point $b_{kn}$ has degree exactly 2.
Let $\phi_k$ denote the resulting branched cover maps.
Now (as in the proof of the converse direction of Proposition \ref{charzn}), equip $S$ with some complex structure and, for all $k$,
pull back the complex structure
to $S_k$ by $\phi_k$.
Equip $M$ with the resulting product complex structure and
define $\phi\,:\, M \,\longrightarrow\, S$ by
$\phi (z_1, \cdots , z_r) \,=\, m(\phi_1(z_1), \cdots , \phi_r(z_r))$. Note that any tuple of the form $(b_{1i_1}, \cdots , b_{ri_r})$
is a non-degenerate critical point of $\phi$. Also $\phi$ is homotopic to $f$. It
follows that $f\,:\, M\,\longrightarrow\, S$ is homotopic to a Lefschetz fibration.
\end{proof}

\section{Examples, Counterexamples and Questions}\label{zoo}

We shall give examples of groups having arbitrarily large, but finite homotopical height.
Several of these groups $G$ admit a closed  manifold as a $K(G,1)$ space but still
have finite homotopical height.

A Theorem that will be used frequently is 
the following strong version of the Lefschetz hyperplane theorem 
due to Goresky and MacPherson \cite[Section 2]{gormac}. 

\begin{theorem}[{\cite[Theorem 8.8]{abckt}}]\label{gmmorse}
Let $V$ be a smooth
(not necessarily closed) complex algebraic variety with $f\,:\, V\,\longrightarrow\, \C
{\mathbb P}^N$ a holomorphic map with finite fibers. If $L\,\subset\, \C{\mathbb P}^N$
is a generic linear subspace,
then $\pi_i(V\, , f^{-1}(L))\,=\,0$ for $i\,\leq\, \dim_\C (f^{-1}(L))$. Equivalently, the
inclusion of $f^{-1}(L)$ in $V$ induces an isomorphism of $\pi_i$ for $i\,<\,
\dim_\C (f^{-1}(L))$ and a surjection for $i \,=\, \dim_\C (f^{-1}(L))$.
\end{theorem}

\subsection{Non-uniform lattices in $CH^n$, $n > 2$} These examples are due
to Toledo \cite{tol1}, who showed that a 
non-uniform lattice $\Ga$ in $CH^n$ is the fundamental group of a smooth projective variety
if and only if $ n \,>\, 2$.
Roughly, the construction is as follows. Toledo constructs an embedding of the manifold $M
\,=\,CH^n/\Ga$ in $\C{\mathbb P}^N$
such that the compactification has just one additional point $P\,\in\, \C{\mathbb P}^N$. A generic hyperplane avoids $P$ and cuts $M$ in a 
(complex) $(n-1)$-dimensional manifold $N$. The universal cover $\til{N}$ of $N$ is Stein
 since $CH^n$ is Stein. Also,   $\pi_i(N) \,=\, \pi_i(M)$ for $i\,<\,n-1$
by Theorem \ref{gmmorse}. Therefore, $ht_\SSS (\Ga) \,\geq\, n-1$, and we have a sequence of
examples of  projective groups of arbitrarily large $\SSS$-homotopical height.

The cohomological dimension of $\Ga$ is $(2n-1)$ as $M$ is a noncompact $K(\Ga,1)$ space.
So $ht_\SSS (\Ga) \,\leq\, 2n-1$.
In fact $\Ga $ is a (2n-1)--dimensional duality group, and so $H^2(\Ga ,\, \Z \Ga)\,=\, 0$.
Further, since $N$ is projective, Proposition \ref{qprojb2} gives that $b_2(\Ga)\,>\,0$.

The special case of a non-uniform lattice $\Ga$ in $CH^3$ is interesting. The above discussion gives
the crude bounds $2\,\leq\, ht_\SSS (\Ga)\,\leq\, 5$. It would be worth evaluating precisely what value $ht_\SSS (\Ga)$ has in this case.

\subsection{Fundamental groups of circle bundles over products of smooth varieties}

Let $M$ and $N$ be two smooth complex projective varieties,
of (complex) dimension $n$, with $\pi_2(M)\,=\,\pi_2(N)\,=\,0$ and $n\,\geq\, 2$. Then
there exists a smooth complex projective variety of dimension $n$
with fundamental group $G$ that fits into an exact sequence
$$
1 \,\longrightarrow\, \Z\,\longrightarrow\, G\,\longrightarrow\,
\pi_1(M) \times \pi_1(N)\,\longrightarrow\, 1\, .
$$
The construction of these examples is again based on Theorem \ref{gmmorse} 
and is detailed in Chapter 8 of \cite{abckt}. The above group $G$ appears as the 
fundamental group of a non-trivial $\C^\ast$--bundle $\LL$ over $M \times 
N$. The homotopy exact sequence of the fibration gives $\pi_i(\LL) \,=\, 
\pi_i(M)\times\pi_i(N)$ for all $i\,\geq\, 2$. In particular,
we have $\pi_2(\LL)\,=\,0$. Also, if both
$M$ and $N$ are aspherical (e.g. if they are products of Riemann surfaces of 
positive genus) then $\LL$ is aspherical and quasi-projective. In such a 
case, $G \,\in\, \QPK$. There exists a complex projective subvariety $W$ of $\LL$ of
complex dimension $n$ such that $\pi_i(W) \,=\, \pi_i(\LL)$ for $i \,<\, n$;
this is explained in Chapter 8 of \cite{abckt}.

If the universal covers $\til{M}$ and $\til{ N}$ of $M$ and $N$ respectively
are both Stein, then so is the universal cover $\til{\LL}$ of $\LL$. Hence
the universal cover $\til{W}$ of $W$ is also Stein. It follows that $ht_\SSS(G)\,\geq\, n$.
Also, the cohomological dimension $cd(G)\,=\, 4n+1$. So we get the bounds
$$
n \,\leq\, ht_\SSS(G) \,\leq\, 4n+1\, .
$$

A special class of these examples, due to Campana \cite{campana}, are the Heisenberg
groups $H^{4n+1}$ given by $$1\,\longrightarrow\, \Z \,\longrightarrow\,
H^{4n+1}\,\longrightarrow\, \Z^{4n} \,\longrightarrow\, 1$$ obtained by taking both
$M, N$ to be abelian varieties. Macinic
points out that $H^{4n+1}$ is $(2n-1)$--stage formal, but not $2n$--stage formal
\cite{macinic}. Hence by Proposition \ref{formalKdim}, we get the better bound
$$n\,\leq\, ht_\SSS(G)\,\leq\, ht_\KK(G)\,\leq\, 2n$$ in this case.

\subsection{K\"ahler versus Symplectic} We give examples of groups that admit a symplectic $K(G,1)$ but not a K\"ahler $K(G,1)$ and upgrade them to give
examples of K\"ahler groups that admit a symplectic $K(G,1)$ but not a K\"ahler $K(G,1)$.

As usual let $H^{2n+1}$ denote the integral Heisenberg group fitting into the exact sequence $$1\,\longrightarrow\, \Z\,\longrightarrow\, H^{2n+1}\,\longrightarrow\, \Z^{2n}
\,\longrightarrow\, 1\, .$$ 
Let $M$ denote the 3-manifold with Nil-geometry having $H^3$ as its fundamental group.
Then $M^3 \times M^3$ admits a symplectic structure. This can be shown directly as follows.
Let $M_i, i=1,2$ denote two copies of $M^3$. Let $p_i\,:\,M_i\,\longrightarrow\, T_i$ denote the natural projections onto tori $T_i$.
Let $\alpha_i, \beta_i  $ denote closed 1-forms generating $H^1(T_i)$, $a_i = p_i^\ast (\alpha_i), b_i = p_i^\ast (\beta_i)$.
  Also let $f_i $ be a closed 1-form in  $H^1(M_i)$ representing the dual to the homology class represented by 
the circle fiber. Let $C_i= f_i\wedge a_i, D_i = f_i\wedge b_i$. Let $\pi_i: M_1 \times M_2 \rightarrow M_i$ denote the natural projections for
$i=1,2$. Abusing notation slightly, we denote $\pi_i^\ast C_i$ by $C_i$ and $\pi_i^\ast b_i$ by $b_i$.

Now, define a 2-form on $M_1 \times M_2$ by $$ \omega = C_1 + C_2 + b_1\wedge b_2.$$ Then a direct computation shows that $\omega^3 = 3 C_1\wedge C_2 \wedge b_1\wedge b_2, $
which is a positive multiple of the volume form and hence $(M_1 \times M_2, \omega)$ is symplectic.
The universal cover of $M_1 \times M_2$ is $\reals^6$ and hence $M_1 \times M_2$ is a $K(H^3\times H^3,1)$.

However, $H^3 \times H^3$ is not K\"ahler (see \cite{macinic} for instance). In particular, there is no K\"ahler  $K(H^3\times H^3,1)$.
Note that by Theorem \ref{morimoto}, $M_1 \times M_2$ admits a complex structure. It follows that the complex and symplectic structures cannot be compatible.

\medskip

We now give an example of a of K\"ahler group that admits a symplectic $K(G,1)$
manifold but not a K\"ahler $K(G,1)$ manifold.
By a result of Campana \cite{campana}, the group $H^{2n+1}$ is K\"ahler for $n>3$. As in the above example let $N_1$ and $N_2$ denote 2 copies of the
$(2n+1)$-dimensional Nil-manifold  having $H^{2n+1}$ as its fundamental group. Then as before $N_1 \times N_2$ is a symplectic $K(G,1)$ for 
the  K\"ahler group $G= H^{2n+1} \times H^{2n+1}$.
If $G= H^{2n+1} \times H^{2n+1}$ admits a K\"ahler $K(G,1)$, then it must be homotopy equivalent to $N_1 \times N_2$. Hence the cohomology of $N_1 \times N_2$
must satisfy the hard Lefschetz property, i.e. there must be an $\omega \in H^2(N_1 \times N_2)$ such that $$\omega^k \cup \,:\, H^{2n+1-k}(N_1 \times N_2) \,
\longrightarrow\, H^{2n+1+k}(N_1 \times N_2)\, $$ is an isomorphism.  But this is not possible by a Theorem of Benson and Gordon \cite{bg}.
Hence $G$ does not admit a K\"ahler $K(G,1)$. In particular, there is no equivariant version of Eliashberg's celebrated Theorem \ref{eli-stein}.

\subsection{C-Symplectic Manifolds}
Recall that c-symplectic (cohomologically symplectic) manifolds are closed manifolds $M^{2n}$ with $\omega \in H^2(M)$
such that $\omega^n \neq 0$. 
Let $CSP$ denote the collection of closed symplectic manifolds with a c-symplectic structure.

\begin{prop} $ht_{CSP} (\Z^{2n-1}) \geq n-1$ for $n>2$. \label{csp-abelian} \end{prop}
\begin{proof}  Let $E$ be an elliptic curve and $E^n$ denote its n-fold product ($n>2$)
embedded in some projective space. 
Let $X$ be a hyperplane section of $E^n$, and let $i\,:\, X \,\hookrightarrow\, E^{2n}$
be the inclusion.
Then $\pi_1(X) = \Z^{2n}$. The natural projection $E^{2n} \,\longrightarrow\, E^{2(n-1)}$ induces a finite map $f\,:\, X \,\longrightarrow\, E^{2(n-1)}$ of smooth complex projective
varieties such that $f^\ast (\omega) = \eta$ for K\"ahler forms $\omega \in H^2(E^{2(n-1)}, \Z)$ and $\eta \in H^2(X, \Z)$. Then 
$f^\ast (\omega^{(n-1)}) = \eta^{(n-1)} \neq 0$ gives a (cohomologically) non-vanishing top form on $X$.

Next $f_\ast\,: \,\pi_1(X)  \,\longrightarrow\, \pi_1(E^{2(n-1)})$ factors through the inclusion isomorphism $i_\ast\,:\, \pi_1(X)\,\longrightarrow\, \pi_1(E^{2n})$
and hence the kernel $K$ is isomorphic to $\Z\oplus \Z$. Attach a 2-handle $D^2 \times S^{2n-4}$ to an $S^1$ in $X$ representing
a primitive element of $K$. Let $Y$ denote the manifold obtained from $X$ by this surgery.  Then $\pi_1(Y) = \Z^{2n-1}$, and
$\pi_i(Y) = \pi_i(X)=0$ for $1<i<n-1$ (the last assertion follows from the Lefschetz hyperplane Theorem and the fact that the surgery
does not introduce homotopy groups in dimension less than $2n-4$). Extending  $f$ over the attaching 2-cell and smoothing it out, we obtain
 a smooth map $g\,:\, Y \,\longrightarrow\,  E^{2(n-1)}$ of positive degree (by the local degree formula for instance).
Hence $g^\ast (\omega^{(n-1)}) \in H^{2(n-1)}(Y)$ is non-zero. Choosing $\alpha = g^\ast (\omega) \in H^{2}(Y)$, we obtain
a c-symplectic structure on $Y$. (See \cite{ikrt} for a closely related set of examples, where the starting point is the symplectic category, rather than
the class of projective manifolds.)
\end{proof}

The class of examples given by Proposition \ref{csp-abelian} can be generalized.

\begin{eg} \label{fin2abvar} Let $W$ be an $n-$ dimensional
smooth variety whose Albanese variety is a smooth  $n-$ dimensional abelian variety $A$.
Let $\phi\,:\, W\,\longrightarrow\, A$ be the (finite) Albanese map. As before, let $X$ be a smooth hyperplane section in $A$.
 We can arrange such that $V=\phi^{-1} (X)$ is a smooth subvariety of
$W$. Also, $H_1(V)\,=\,H_1(A)\,=\, \Z^{2n}$. 

The natural {\it smooth} projection $A\,\longrightarrow\, T^{2(n-1)}$ induces a finite map $f\, :\, V \,\longrightarrow\, T^{2(n-1)}$, which automatically has
positive degree. Then, as before, $f_\ast\,:\, H_1(V)\,\longrightarrow\, H_1(T^{2(n-1)})$ has kernel $K$ of rank two.
Attach, as before, a 2-handle $D^2 \times S^{2n-4}$ to an $S^1$ in $V$ representing
a non-torsion element of $K$. Let $Y$ denote the manifold obtained from $V$ by this surgery. Then $H_1(Y)$ has rank  ${2n-1}$. 
Extend  $f$ over the attaching 2-cell as before and smooth it out to obtain
 a smooth map $g\,:\, Y\,\longrightarrow\, E^{2(n-1)}$ of positive degree. Choosing $\alpha = g^\ast (\omega) \in H^{2}(Y)$, we obtain
a c-symplectic structure on $Y$.

If in addition $V$ is constructed in such a way that $\til{V}$ is homotopy equivalent to a wedge of $(n-1)$-spheres, then $\pi_i(Y) = \pi_i(V)=0$ for $1<i<n-1$
as before and hence $ht_{CSP}(\pi_1(Y)) \geq n-1$.
\end{eg}

\subsection{Further Questions}
We end with some further questions which we were led to during the course of our study. 
The following is a  special case of the Shafarevich conjecture.
\begin{qn} If a smooth projective variety
$X$ is a $K(G,1)$, is $\til X$ Stein? This is true for $\til X$ a product of bounded domains and affine spaces.\end{qn}

\begin{qn}
Let $ht_\SSS(G) = n$, and $M$ be a  smooth projective variety of (complex) dimension $m$ realizing the $\SSS-$homotopical height of $G$.
If in addition the action of $G (= \pi_1(M)$ on $\pi_n(M)$ is the trivial action, and
$H^n(G,\ZG)\neq 0$ is $G$  virtually  a PD(n) group, with $n=2m$?
\end{qn}

\section*{Acknowledgments}

We would like to thank Domingo Toledo for discussions
that led us to investigate $H^i(G,\, \ZG)$ for $i\,>\,1$. In particular, he
urged us to look at $H^2(G,\,\ZG)$. Question \ref{tol-h2} has emerged from
these discussions. We also thank R. V. Gurjar for Lemma \ref{gurjar}, Samik Basu for the reference \cite{macinic}
and Alex Suciu for the reference \cite{suciu-imrn}.

\end{document}